\documentclass[11pt]{article}

\usepackage{amssymb,amsmath, amsthm}
\usepackage[width=17cm,height=21cm]{geometry}
\usepackage{textcomp, gensymb}
\usepackage[utf8]{inputenc}
\usepackage[title]{appendix}
\usepackage{color}
\usepackage{todonotes}
\usepackage{graphicx}
\usepackage{pdfpages}
\usepackage{hyperref}
\usepackage{mathtools}
\usepackage[capitalise]{cleveref}
\usepackage{enumitem}
\usepackage{nicefrac}
\usepackage{stmaryrd}
\usepackage{authblk}
\usepackage{float}
\usepackage{multicol}
\usepackage{subcaption}
\usepackage{tikz}
\usepackage[ruled,vlined]{algorithm2e}
\usepackage[pagewise]{lineno}

\newcommand{\R}{\mathbb{R}}
\newcommand{\N}{\mathbb{N}}

\newcommand{\e}{\varepsilon}
\newcommand{\di}[1]{\,\mathrm{d}#1}
\newcommand{\dive}{\operatorname{div}}
\newcommand{\dist}{\operatorname{dist}}
\newcommand{\tr}{\operatorname{Tr}}

\newcommand{\inside}{\operatorname{inside}}
\newcommand{\ddt}{\frac{\operatorname{d}}{\operatorname{d}t}}

\newcommand{\ID}{\operatorname{Id}}

\newcommand\restr[2]{{
  \left.\kern-\nulldelimiterspace 
  #1 
  \vphantom{\big|} 
  \right|_{#2} 
  }}

\makeatletter
\newcommand{\myitem}[1]{%
\item[#1]\protected@edef\@currentlabel{#1}%
}
\makeatother

\newtheorem{definition}{Definition}
\newtheorem{theorem}{Theorem}

\newtheorem{lemma}{Lemma}
\newtheorem{remark}{Remark}

\crefname{lemma}{Lemma}{lemmas}
\Crefname{lemma}{Lemma}{Lemmas}
\crefname{thm}{theorem}{theorems}
\Crefname{thm}{Theorem}{Theorems}
\Crefname{algocf}{Algo.}{Algorithm}

\numberwithin{equation}{section}

\begin{document}

\title{Precomputing approach for a two-scale phase transition model}
\author[$\dagger,1$]{Michael Eden}
\author[$\star$]{Tom Freudenberg}
\author[$\dagger$]{Adrian Muntean}

\affil[$\dagger$]{Department of Mathematics and Computer Science, Karlstad University, Sweden}
\affil[$\star$]{Center for Industrial Mathematics, University of Bremen, Germany}

\maketitle
\begin{abstract}
In this study, we employ analytical and numerical techniques to examine a phase transition model with moving boundaries.
The model displays two relevant spatial scales pointing out to a macroscopic phase and a microscopic phase, interacting on disjoint  inclusions.
The shrinkage or the growth of the inclusions is governed by a modified Gibbs-Thomson law depending on the macroscopic temperature, but without accessing curvature information.
We use the Hanzawa transformation to transform the problem onto a fixed reference domain.
Then a fixed-point argument is employed to demonstrate the well-posedness of the system for a finite time interval. 
Due to the model's nonlinearities and the macroscopic parameters, which are given by differential equations that depend on the size of the inclusions, the problem is computationally expensive to solve numerically.
We introduce a precomputing approach that solves multiple cell problems in an offline phase and uses an interpolation scheme afterward to determine the needed parameters.
Additionally, we propose a semi-implicit time-stepping method to resolve the nonlinearity of the problem.
We investigate the errors of both the precomputing and time-stepping procedures and verify the theoretical results via numerical simulations.
\end{abstract}



\footnotetext[1]{\url{michael.eden@kau.se}}
\stepcounter{footnote}
\section{Introduction}\label{sec:introduction}
Phase transition models describe the heat dynamics in a system where phase transitions can occur, such as water-ice or austenite-pearlite microstructures in steel.
One of the main difficulties is the need to account for the changing geometries inherent to the transition process.
In some cases, there is an additional scale separation: the phase transition processes might happen at the microscale while our interest might be in the effective macroscopic behaviors.
Examples include microstructural changes in steel \cite{Eden_Muntean_thermo,Vander2004}, liquid–solid phase transitions in metals \cite{EKK02,Kobayashi1993}, or thawing processes in permafrost \cite{Peszynska2024}. In these cases, the focus is not on the geometric changes at the microscale but on their impact on macroscopic properties.
However, via homogenization techniques (e.g., two-scale convergence or asymptotic expansion) it is easy to see that the microstructural changes can not be neglected: In the resulting models, the phase transitions are still present on the microscale and directly influence the effective material parameters, see, e.g., \cite{EdenMuntean24,EKK02,Visintin1997}.
Such two-scale models (sometimes also called micro-macro model or distributed microstructures model) are computationally extremely costly, mainly due to the following two factors: 
\begin{enumerate}
    \item[$(a)$] \textit{Strong non-linear coupling between the scales:} The macroscopic temperature field is driving the microstructural changes and the microstructural changes are in turn non-linearly influencing the macroscopic material parameters.
    \item[$(b)$] \textit{Family of moving boundary problems:} Every macroscopic point is associated with corresponding moving boundary problems at the microscale both via the interaction with the microscopic temperature but also via the geometric changes influencing the heat conductivity via the associated cell problems.
\end{enumerate} 
Since one of the main goals of homogenization is the derivation of simpler models that are computationally less heavy, these issues might at first glance call into question the usefulness of such limit models.
There is, however, one important advantage these two-scale models have, namely the explicit separation of the scales.
This can be exploited in many different ways: via iteration procedures \cite{GärttnerFrolkovicKnabner2020}, parallel computing \cite{Bastidas2020}, machine learning \cite{Gärttner2022}, and/or interpolation techniques \cite{Nepal2024}.

In this work, we propose, analyze, and implement a precomputing strategy in which the calculations of the macroscopic coefficients are pushed into an \textit{offline phase}.
Instead of having to calculate the effective conductivity via geometry-dependent cell problems at every macroscopic point and at each time step, we precompute the conductivity instead for a range of potential parameter values and interpolate in between.
Please note, that this precomputing is \textit{perfectly parallizable}.
We show that this interpolation is stable and that the introduced error in the two-scale system can be controlled via the step size in the parameter space and the order of the interpolation. 
This idea is not completely new and similar approaches have been outlined before, see \cref{subsec:lit_review} for a short literature overview.
The non-linearities in the model are resolved with a semi-implicit time-stepping method which we show to converge linearly.
We confirm these convergence results via simulations and the corresponding FEniCS code is made freely available \cite{Code}.

\subsection{Two-scale phase transition models}
We want to elaborate a bit on the origin and structure of the two-scale phase transition problem that we are looking at in this work.
The model in question is the homogenization limit of a highly heterogeneous two-phase system consisting of a connected phase $\Omega_\e^a(t)$ with disconnected inclusions of a different phase $\Omega_\e^b(t)$, see~\cref{figure:domain_evolution}. 
Both domains depend on time due to the potential of phase transitions. 
Here, $\e>0$ is a small parameter denoting the heterogeneity of the system and the homogenization limit can be achieved via a (formal or rigorous) limit process $\e\to0$.
\begin{figure}[ht]
\centering
  \vspace{-.3cm}
  \scalebox{.35}{\input{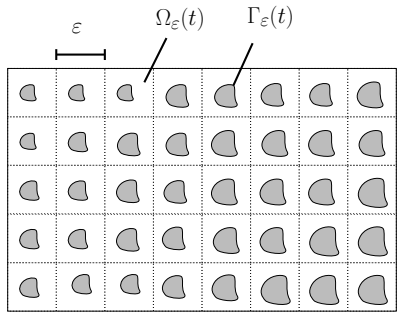}}
  \caption{A potential setup of a highly heterogeneous two-phase system with evolving phase interface.}
  \label{figure:domain_evolution}
\end{figure}

\begin{subequations}\label{system:e_problem}
A two-phase problem for such a phase transition problem could look like
\begin{alignat}{2}
    \partial_t\theta_\e^a-k^a\Delta \theta_\e^a&=f_\e^a&\quad&\text{in}\ \Omega_\e^a(t),\ t>0,\\
    \partial_t\theta_\e^b-\e^2k^b\Delta \theta_\e^b&=f_\e^b&\quad&\text{in}\ \Omega_\e^b(t),\ t>0,\\
    \theta_\e^a&=\theta_\e^b&\quad&\text{on}\ \Gamma_\e(t),\ t>0,\label{eq:eps_continuity}\\
    -(k^a\nabla\theta_\e^a+\e^2k^b\nabla\theta_\e^b)\cdot n_\e&=\e Lv_\e&\quad&\text{on}\ \Gamma_\e(t),\ t>0,
\end{alignat}
where $L\in\R$ denotes the Latent heat and $v_\e$ the normal velocity of the moving phase interface $\Gamma_\e(t)$.
The particular $\e$-scalings are natural for this particular setup with inclusions, cf. \cite{ArbogastDouglasHornung90,EdenMuntean24,Gahn21}.
The normal velocity has to be implicitly or explicitly described via an additional equation.
The usual Stefan condition is $\theta_\e^a=\theta_{ref}$, where $\theta_{ref}$ is the reference temperature of the phase transition in question.\footnote{In a water/ice system under normal conditions, we have $\theta_{ref}=0^\circ C$.}
Another option is Gibbs-Thomson law without curvature which postulates
\begin{equation}
v_\e=\e (\theta_\e^a-\theta_{ref})\quad\text{on}\ \Gamma_\e(t),\ t>0.
\end{equation}
\end{subequations}
The model in question in this work is the limit of System \eqref{system:e_problem} for $\e\to0$.
Such a limit procedure was conducted via two-scale convergence techniques for a slightly simplified one-phase equivalent of this system in \cite{EdenMuntean24}, and the results can easily, at least formally, be extended to this two-phase system via asymptotic expansion.
Via such a limit process, we can identify limits such that, in some sense, $\theta_\e^a\to\Theta$ and $\theta_\e^b\to\theta$.
Here, $\Theta$ is the temperature field of the macroscopic phase which is defined on a non-changing macroscopic domain $\Omega$ and $\theta$ is the temperature field of the microscopic phase given on a changing microscopic domain $Y(t)$.
Of course, $\Theta$ and $\theta$ are intricately coupled via energy conservation and the continuity condition \cref{eq:eps_continuity}.
Moreover, all material coefficients for the macroscopic heat problem are now dependent on the microstructural changes as well, rendering the problem highly nonlinear.
The exact setup of the two-scale problem is given in \cref{sec:setting_model}.

\subsection{Discussion on related literature}\label{subsec:lit_review}
An early work on two-scale models in phase transitions is due to A.~Visintin \cite{Visintin1997}, where the transition between the two scales is achieved via convolution with a Gaussian kernel but without explicitly looking at evolving microstructures.
More closely related to our work, in \cite{EKK02,Eck2005}, two-scale models for dendritic growth with prescribed geometry changes are formally derived via asymptotic expansion.
In \cite{Eden_Muntean_thermo, eden_homogenization_2019} two-phase models for phase transitions are derived, albeit with a priori known geometry evolution.
Very recently, in \cite{EdenMuntean24}, an effective one-phase model for phase transitions was rigorously derived via two-scale convergence.

Very similar models and challenges can be found in the context of reactive transport or coagulation/fragmentation dynamics in evolving porous media which has garnered quite a bit of interest in recent years.
We point to the upscaling works in \cite{Bhattacharya22,Gahn21,GahnPop23,MunteanNikolopolous20,RayNoordenFrankKnabner12,WiedemannPeter23}, which are concerned with the derivation of two-scale models in that context, and note that the resulting models are structurally quite similar to the phase transition problem considered in this work.
Naturally, there are quite a few works looking more closely at the resulting two-scale models: In \cite{Meier09}, the well-posedness for such a two-scale model was established for a simplified setup where the geometry changes do not impact the diffusivity except via the porosity.
In \cite{GaerttnerKnabnerRay23}, strong local-in-time solutions for a two-scale model, even allowing for nonuniform geometry changes, are established via maximal parabolic regularity and a level set approach.

Such models often come with huge computational challenges due to their nonlinear nature and complex coupling.
Our main approach of dealing with this is a precomputing strategy in which the calculations of the effective macroscopic coefficients are pushed into an offline phase; introducing an interpolation error but significantly speeding up the calculations.
Similar ideas are often explored in the context of shape or topology optimization, see e.g.,  \cite{LichtiLeichner2022,Clezio23,VuLukesStingl2023} and references therein.
In \cite{Nepal2024}, the convergence of a precomputing scheme for a two-scale dispersion model for porous media is shown, however, there are no geometry changes present.

Naturally, there are other approaches to tackle the computational complexity of two-scale models:
Instead of solving the entire coupled system in a classical FE approach, an alternative idea is to explicitly decouple the macro and micro temperatures by solving between them iteratively.
This is closely related to the FE$^2$ approach \cite{FEYEL2003} and the Heterogeneous Multiscale Method \cite{WeinanEnggqvist03}.
This approach has two main advantages: the calculations on the microscale can be run in parallel and an adaptive interpolation strategy can be used to only solve the microscale problems for a small number of macroscopic points.
However, for each time step, there may be a need for several iteration steps, which again can be costly, especially when the coupling is complex.
In addition, the iteration process introduces a new error.
The approach has been studied and investigated for many different two-scale problems, see \cite{Bastidas2020,ElbingerKnabner2022,freudenberg2024,LindMuntean18,NeussRaduLudwigJager10}.
As a direct extension of the iterative scheme, an adaptive clustering approach has been developed in \cite{REDEKER2013}.
Here, micro problems with comparable macroscopic data are paired together and solved via one representative system.
This allows for a speed-up and even enables fast simulations in higher-dimensional setups.
Of course, this approach produces further error sources, for which, to our knowledge, no estimates are available.
But the numerical results in \cite{REDEKER2013} show good approximation accuracy. 

Another approach is the hierarchical FEM as they are used for multiscale problems in \cite{Hoang2004, PARK2020}.
Here, special FE spaces are constructed to reduce the DOFs of those functions which depend on multiple scales.
This method is comparable to a wavelet approach and allows for analytical error estimates, but is more complicated to set up than the previous methods.
Finally, we want to mention a model reduction approach which is often encountered in mechanics, see \cite{GUO2024} and the references therein.
A model reduction for microscopic crystal growth was considered in \cite{Redeker2015}. 
This approach can lead to significant speed-ups, but requires, especially for time-dependent problems, a large amount of data to obtain reasonable models.

\subsection{Outline of the paper}
The paper is structured as follows: In \cref{sec:setting_model}, we present the two-scale system, collect the mathematical assumptions, and introduce our notion of weak solutions.
In \cref{sec:analysis}, we establish the existence of local-in-time weak solutions via a fixed-point argument by utilizing Hanzawa transformations (\cref{thm:existence}).
Moreover, we show that uniqueness and long-time existence are given when some additional estimates are met.
This is followed up by \cref{sec:numerical_analysis}, where we introduce and discuss our precomputing strategy.
Here, we first show the stability of the resulting scheme with respect to the interpolation error (\cref{thm:error_precomputing}) before then analyzing our semi-implicit time-stepping method with which we resolve the non-linearities (\cref{lemma:time_stepping}).
We present some simulation scenarios in \cref{sec:simulations}: in particular, we are showing convergence curves for our precomputing strategy.
Finally, we close with a short conclusion and discussion (\cref{section:conclusion}).

\section{Setting of the geometry and model equations}\label{sec:setting_model}
In this section, we introduce the geometric setup and the two-scale system modeling a two-phase composite medium with phase transformations taking place at the microscale.
We also collect the mathematical assumptions and present our notion of weak convergence.

To this end, let $S = (0, T )$, for $T\in\R\cup\{\infty\}$, denote the time interval of interest.
Let a Lipschitz domain $\Omega\subset\R^d$ ($d=2,3$, outside normal vector $\nu$) represent the macroscopic domain, set $Y = (0, 1)^d$ and let a $C^3$-domain $Y_0\subset\subset Y$ (outer normal vector $n$) represent the initial micro-domain.
We set $\Gamma_0=\partial Y_0$.
The micro-domain $Y_0$ is allowed to grow or shrink based on the dynamics of the problem and this change is characterized via a height parameter $h$ inside a tubular neighborhood $U_\Gamma\subset\subset Y$ of $\Gamma_0$ with width $2a$ for some $a>0$.
Here, $h=0$ indicates no changes, $h<0$ shrinkage, and $h>0$ growth.
We denote the dynamic domain by $Y(h)$ and interface by $\Gamma(h)$,\footnote{Of course, $Y(0)=Y_0$ and $\Gamma(0)=\Gamma_0$.} respectively, via
\[
\Gamma(h):=\left\{\gamma+h n(\gamma)\ : \ \gamma\in\Gamma_0\right\},\quad Y(h):=\inside(\Gamma(h)).
\]
Please note that the growth of the micro-domain is cell uniform, i.e., the height $h=h(t,x)$ only depends on time $t\in S$ and the macro variable $x\in\Omega$.
As a consequence, shape changes are not possible (e.g., if $Y_0$ is a ball of radius $r$, $Y(h)$ will be a ball of radius $r+h$).
For the exact setup of the transformation, we refer to \cref{subsec:hanzawa}, and, for a sketch of this setup, we refer to \cref{fig:two-scale_setup}.
\begin{figure}[ht]
    \centering
  \scalebox{1.0}{\tikzset{every picture/.style={line width=0.75pt}} 

\begin{tikzpicture}[x=0.75pt,y=0.75pt,yscale=-1,xscale=1]

\draw  [color={rgb, 255:red, 213; green, 213; blue, 213 }  ,draw opacity=1 ][fill={rgb, 255:red, 255; green, 255; blue, 255 }  ,fill opacity=1 ][dash pattern={on 4.5pt off 4.5pt}] (439.94,31.46) .. controls (439.94,23.17) and (446.66,16.46) .. (454.94,16.46) .. controls (463.23,16.46) and (469.94,23.17) .. (469.94,31.46) .. controls (469.94,39.74) and (463.23,46.46) .. (454.94,46.46) .. controls (446.66,46.46) and (439.94,39.74) .. (439.94,31.46) -- cycle ;
\draw   (209,41.25) .. controls (209,31.17) and (217.17,23) .. (227.25,23) -- (378.17,23) .. controls (388.25,23) and (396.42,31.17) .. (396.42,41.25) -- (396.42,96) .. controls (396.42,106.08) and (388.25,114.25) .. (378.17,114.25) -- (227.25,114.25) .. controls (217.17,114.25) and (209,106.08) .. (209,96) -- cycle ;
\draw  [fill={rgb, 255:red, 0; green, 0; blue, 0 }  ,fill opacity=1 ] (328.5,34) .. controls (328.5,33.17) and (329.17,32.5) .. (330,32.5) .. controls (330.83,32.5) and (331.5,33.17) .. (331.5,34) .. controls (331.5,34.83) and (330.83,35.5) .. (330,35.5) .. controls (329.17,35.5) and (328.5,34.83) .. (328.5,34) -- cycle ;
\draw  [fill={rgb, 255:red, 0; green, 0; blue, 0 }  ,fill opacity=1 ] (316.5,96) .. controls (316.5,95.17) and (317.17,94.5) .. (318,94.5) .. controls (318.83,94.5) and (319.5,95.17) .. (319.5,96) .. controls (319.5,96.83) and (318.83,97.5) .. (318,97.5) .. controls (317.17,97.5) and (316.5,96.83) .. (316.5,96) -- cycle ;
\draw   (427.09,3.6) -- (482.8,3.6) -- (482.8,59.31) -- (427.09,59.31) -- cycle ;
\draw  [fill={rgb, 255:red, 255; green, 255; blue, 255 }  ,fill opacity=1 ] (446.17,31.46) .. controls (446.17,26.61) and (450.1,22.68) .. (454.94,22.68) .. controls (459.79,22.68) and (463.72,26.61) .. (463.72,31.46) .. controls (463.72,36.3) and (459.79,40.23) .. (454.94,40.23) .. controls (450.1,40.23) and (446.17,36.3) .. (446.17,31.46) -- cycle ;
\draw    (330,34) .. controls (361.12,12.51) and (391.92,15.31) .. (422.32,23.31) ;
\draw   (497.75,62.27) -- (553.47,62.27) -- (553.47,117.98) -- (497.75,117.98) -- cycle ;
\draw  [fill={rgb, 255:red, 255; green, 255; blue, 255 }  ,fill opacity=1 ] (503.75,90.12) .. controls (503.75,78.05) and (513.54,68.26) .. (525.61,68.26) .. controls (537.68,68.26) and (547.47,78.05) .. (547.47,90.12) .. controls (547.47,102.2) and (537.68,111.99) .. (525.61,111.99) .. controls (513.54,111.99) and (503.75,102.2) .. (503.75,90.12) -- cycle ;
\draw    (318,96) .. controls (335.42,128.25) and (461.42,132.25) .. (489.03,92.8) ;
\draw    (462.8,28.06) -- (498.08,12.5) ;
\draw    (454.94,31.46) -- (510.08,41.5) ;
\draw    (545.95,83.29) -- (575.58,71.5) ;
\draw   (107.75,38.62) -- (163.47,38.62) -- (163.47,94.33) -- (107.75,94.33) -- cycle ;
\draw  [color={rgb, 255:red, 213; green, 213; blue, 213 }  ,draw opacity=1 ][fill={rgb, 255:red, 255; green, 255; blue, 255 }  ,fill opacity=1 ] (120.61,66.48) .. controls (120.61,58.19) and (127.33,51.48) .. (135.61,51.48) .. controls (143.89,51.48) and (150.61,58.19) .. (150.61,66.48) .. controls (150.61,74.76) and (143.89,81.48) .. (135.61,81.48) .. controls (127.33,81.48) and (120.61,74.76) .. (120.61,66.48) -- cycle ;
\draw    (135.61,66.48) -- (146.95,34.45) ;
\draw  [color={rgb, 255:red, 213; green, 213; blue, 213 }  ,draw opacity=1 ][fill={rgb, 255:red, 255; green, 255; blue, 255 }  ,fill opacity=1 ][dash pattern={on 4.5pt off 4.5pt}] (510.61,90.12) .. controls (510.61,81.84) and (517.33,75.12) .. (525.61,75.12) .. controls (533.89,75.12) and (540.61,81.84) .. (540.61,90.12) .. controls (540.61,98.41) and (533.89,105.12) .. (525.61,105.12) .. controls (517.33,105.12) and (510.61,98.41) .. (510.61,90.12) -- cycle ;
\draw    (525.61,90.12) -- (581.58,104.33) ;
\draw    (100.37,55.36) -- (120.44,63.96) ;

\draw (216.5,28.9) node [anchor=north west][inner sep=0.75pt]  [font=\normalsize]  {$\Omega $};
\draw (317.33,36.07) node [anchor=north west][inner sep=0.75pt]  [font=\small]  {$x_{1}$};
\draw (303,98.07) node [anchor=north west][inner sep=0.75pt]  [font=\small]  {$x_{2}$};
\draw (501.33,2.9) node [anchor=north west][inner sep=0.75pt]    {$\Gamma ( h( t,x_{1}))$};
\draw (512.83,33.4) node [anchor=north west][inner sep=0.75pt]    {$Y( h( t,x_{1}))$};
\draw (582.83,98.4) node [anchor=north west][inner sep=0.75pt]    {$Y( h( t,x_{2}))$};
\draw (578.83,63.9) node [anchor=north west][inner sep=0.75pt]    {$\Gamma ( h( t,x_{2}))$};
\draw (143.43,15.4) node [anchor=north west][inner sep=0.75pt]    {$Y_{0}$};
\draw (84.77,34.4) node [anchor=north west][inner sep=0.75pt]    {$\Gamma _{0}$};

\end{tikzpicture}}
  \caption{Representation of the two-scale system at some time $t>0$. At the left an example for the initial micro-domain $Y_0=Y(0)$. In this picture, $h(t,x_1)<0$ (shrinking ball) and $h(t,x_2)>0$ (growing ball).
  Please note that the mathematical setup and the theoretic result allows for any $C^3$-inclusion (not just balls) but in our numerical simulations, we only look at balls.}
  \label{fig:two-scale_setup}
\end{figure}
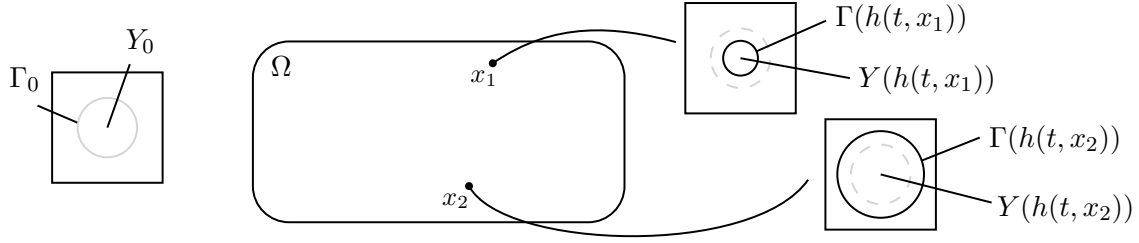

Now, let $\Theta=\Theta(t,x)$ and $\theta=\theta(t,x,y)$ denote the macroscopic and microscopic temperature, respectively, and $h=h(t,x)$ the height function characterizing the micro-domain changes.
On the macroscale, we have a standard heat equation where all parameters non-linearly depend on $h$:
\begin{subequations}\label{eq:system_PDE}
\begin{alignat}{2}
\partial_t(C(h)\Theta)-\dive(K(h)\nabla\Theta)&=F- L(h)\partial_th-q(h)[\theta]&\quad&\text{in}\ S\times\Omega,\label{eq:system:macro_heat}\\
-K(h)\nabla\Theta\cdot n&=0&\quad&\text{on}\ S\times\partial\Omega\label{eq:system:macro_heatb},\\
\Theta(0)&=\Theta_0&\quad&\text{in}\ \Omega\label{eq:system:macro_heatc}.
\end{alignat}
\end{subequations}
Here, $C$ denotes the heat capacity, $K$ the effective heat conductivity, and $ L$ is the latent heat.
Moreover, $q(h)[\theta]\colon S\times\Omega\to\R$ accounts for the energy flow from the micro into the macro region and $F$ is a generic volume source.
The $h$-dependent coefficients are given via
\begin{subequations}
\begin{alignat}{2}
    C&\colon(-a,a)\to(0,\infty),&\quad C(h)&=1-|Y(h)|,\label{eq:system:heat_cap}\\
    K&\colon(-a,a)\to\R^{d\times d},&\quad K(h)_{i,j}&=\int_{Y\setminus Y(h)}\mathcal{K}(\nabla_y\xi_j(h)+e_j)\cdot(\nabla_y\xi_i(h)+e_i)\di{y},\label{eq:system:conductivity}\\
     L&\colon(-a,a)\to\R,&\quad  L(h)&=|\Gamma(h)|,\label{eq:system:latent_heat}\\
    q(\cdot)[\theta]&\colon(-a,a)\to \R,&\quad q(h)[\theta]&=\int_{\Gamma(h)}\kappa\nabla_y\theta\cdot n\di{\sigma_y}\label{eq:system:flux}.
\end{alignat}
Here, $\mathcal{K}, \kappa>0$ are the heat conductivity of the macro- and micro-phase and $\xi_j(h)\in H^1_\#(Y\setminus Y(h))$, $j=1,...,d$, solutions to the cell problems
\begin{alignat}{2}\label{eq:system:cell_problem}
\int_{Y\setminus Y(h)}\mathcal{K}(\nabla_y\xi_j(h)+e_j)\cdot\nabla_y\phi\di{y}=0
\end{alignat}
for all $\phi\in H^1_\#(Y\setminus Y(h))$.
\end{subequations}

Please note that the coefficients \cref{eq:system:heat_cap,eq:system:conductivity,eq:system:latent_heat,eq:system:flux} represent what we typically expect in a two-scale limit system; see \cite{Bastidas2020,EKK02,EdenNikolopoulosMuntean22,KNP14} for similar setups.
For example, $C(h)$ represents the effective volumetric heat capacity of the macroscopic phase which is given by $C(h)=|Y\setminus Y(h)|=1-|Y(h)|$.
This accounts for the competition with the micro-phase.

The heat problem for the macroscopic phase (\cref{eq:system:macro_heat,eq:system:macro_heatb,eq:system:macro_heatc}) is strongly coupled with a Stefan-type problem for the microscopic phase in the form of
\begin{subequations}
\begin{alignat}{2}
\partial_t\theta-\dive_y(\kappa\nabla_y\theta)&=f&\quad&\text{in}\ Y(h), \ \text{in}\  S\times\Omega,\\
\theta&=\Theta&\quad&\text{on}\ \Gamma(h), \ \text{in}\  S\times\Omega,\label{eq:system:continuity}\\
\partial_th&=\theta-\Theta^{ref}&\quad&\text{on}\ \Gamma(h), \ \text{in}\ S\times\Omega,\label{eq:system:gibbs_thomson}\\
\theta(0)&=\theta_0&\quad&\text{in}\ \Omega\times Y(0).
\end{alignat}
\end{subequations}
The moving boundary condition \eqref{eq:system:gibbs_thomson} is a modified \textit{Gibbs-Thomson} law without curvature, where the normal velocity of the moving boundary, i.e., $\partial_th$, is determined by the difference of the temperature to a reference temperature.
As an example, for the water-ice system under normal conditions, we have $\Theta^{ref}=0^\circ C$.
Here, we just assume $\Theta^{ref}\in \R$.
The overall problem is both highly nonlinear due to the changing domains and is fully coupled via the following mechanism:
\begin{itemize}
    \item The flux term $q$ in \cref{eq:system:macro_heat} tracks the heat exchange between the phases. As a consequence, the microphase can act as a heat storage or heat sink.
    \item \Cref{eq:system:continuity} enforces continuity of the micro and macro temperatures at the micro-interface $\Gamma(h)$. This also implies cell-uniform growth: Since $\Theta$ is $y$-independent, $h$ must be as well.
    \item The growing or shrinking domains directly influence the macroscopic coefficients (\cref{eq:system:heat_cap,eq:system:conductivity,eq:system:latent_heat,eq:system:flux}).
\end{itemize}
We point out that the Neumann flux appearing via $q(h)[\theta]$ in the strong form \cref{eq:system:macro_heat} disappears in the weak form as it balances with the corresponding term in the micro part after integration by parts.
The main difficulties in the analysis of this system are the changing micro-geometries and the inherent non-linearity associated with those changes.
We do this by transforming the geometries into the fixed, initial geometries via the so-called Hanzawa transformation.
Details are given in \cref{subsec:hanzawa}.

\begin{remark}\label{remark:initial_height}
    In our setup, it holds $h(0)=0$ which means that the initial micro-geometry is uniform in $x\in\Omega$.
    However, we could easily consider an initial condition $h\in L^\infty(\Omega)$ as long as it is small enough (see also \cref{remark:analysis_initial_height}).
    As this would introduce some additional technical estimates in our analysis, we stick to a uniform initial setup. 
\end{remark}

\subsection{Assumptions and weak forms}
A few words, regarding our notation: In general, we use lowercase and capitalized symbols for the microscopic and macroscopic counterparts, e.g., $f$ and $F$ for the volume source densities.
For functions depending on both the macroscopic variable $x\in\Omega$ and $y\in Y$, we take $\dive_y,\nabla_y$ to denote the spatial differential operators with respect to $y$.
Moreover, we take $c,C>0$ as generic constants (whose value can change even from line to line) and where we explicitly point out when they are independent of a specific parameter.
We start by formulating the assumptions we place on our geometry and data:
\begin{enumerate}
    \item[\textbf{(A1)}]\label{assumptionA1} The heat conductivities $\kappa,\mathcal{K}\in\R^{d\times d}$ are symmetric and positive definite.
    \item[\textbf{(A2)}]\label{assumptionA2} The volume source densities satisfy $f\in L^\infty(S\times\Omega;C^{0,1}(\overline{Y_0}))$ and $F\in L^\infty(S\times\Omega)$.
    \item[\textbf{(A3)}]\label{assumptionA3} The initial geometry $Y(0)=Y_0\subset\subset(0,1)^d$ is a $C^3$ domain and $\Omega\subset\R^d$ is a Lipschitz domain. 
    The tubular neighborhood $U_\Gamma\subset(0,1)^d$ of the interface $\Gamma=\partial Y_0$ has tubular radius $a>0$.
    \item[\textbf{(A4)}] The initial conditions satisfy $\theta_{0}\in L^\infty(\Omega\times Y_0)$ and $\Theta_0\in L^\infty(\Omega)$.
\end{enumerate}
Please note that the regularity presented in Assumptions $(A2)$ and $(A4)$ is higher than the standard parabolic regularity for two reasons: To be able to work with coordinate transforms, the micro heat production source needs to be continuous in the $y$-Variable.
We chose Lipschitz continuous to guarantee the uniqueness of solutions.
The boundedness of the data is needed to ensure boundedness of the solutions; without this, it is impossible to control the height functions (see \cref{eq:system:gibbs_thomson}).

Now, for any $0<T^*\leq T$ and $h\in L^\infty((0,T^*)\times\Omega)$ satisfying $\|h\|_\infty\leq a^*=\nicefrac{a}{8}$, we introduce the corresponding non-cylindrical micro-domain
\[
Q[h](T^*)=\bigcup_{(t,x)\in (0,T^*)\times\Omega} \Bigl[\{(t,x)\}\times Y(h(t,x)) \Bigr]\subset (0,T')\times\Omega\times Y.
\]
and the Banach space
\begin{multline*}
\mathcal{X}[h](T^*)=\big\{(\Theta,\theta)\in L^2(0,T^*;H^1(\Omega))\times L^2(Q[h](T^*))\ : \
\theta(t,x)\in H^1(Y(h)),\\ \theta(t,x,y)=\Theta(t,x)\ \text{on}\ \Gamma(h)\ \text{for a.a. }(\tau,x)\in (0,T')\times\Omega,\ y\in \Gamma(h) \big\}.
\end{multline*}
Assuming that the height function $h\in L^\infty((0,T^*)\times\Omega)$ describes the moving boundaries and taking any function $(\Phi,\phi)\in\mathcal{X}[h](T^*)$, we introduce the weak form for the PDE system \eqref{eq:system_PDE}:
\begin{multline*}
\int_{\Omega}\partial_t(C(h)\Theta)\Phi\di{x}+\int_\Omega\int_{Y(h)}\partial_t\theta\phi\di{y}\di{x}+\int_{\Omega}K(h)\nabla\Theta\cdot\nabla\Phi\di{x}
+\int_\Omega\int_{Y(h)}\kappa\nabla_y\theta\cdot\nabla_y\phi\di{y}\di{x}\\
=\int_{\Omega}\left(F- L(h)\partial_th\right)\Phi\di{x}
+\int_\Omega\int_{Y(h)}f\phi\di{y}\di{x}
\end{multline*}
Please note that the $q(h)[\theta]$ term in \cref{eq:system:macro_heat} cancels with the micro boundary integral due to $\phi=\Phi$ on $\Gamma(h)$.
With that in mind, we can now define a solution of the two-scale heat problem in moving coordinates:

\begin{definition}[Weak solution in moving coordinates]\label{def:weak_sol_moving}
We say that $(h,\Theta,\theta)$ is a local-in-time weak solution in moving coordinates, when there is a time horizon $S^*=(0,T^*)$, such that

\begin{itemize}
    \item[$(i)$] The domain changes are characterized by the height function $h\in W^{1,\infty}(S^*;L^\infty(\Omega))$ which satisfy
    \[
    \|h\|_{L^\infty(S^*\times\Omega)}\leq a^*.
    \]
    \item[$(ii)$] The velocity of the domain evolution is driven by the temperature:
    \[
    \partial_th=\Theta-\Theta^{ref}\  \text{in}\ L^2(S^*\times\Omega).
    \]
    \item[$(iii)$] The temperatures $(\Theta,\theta)\in \mathcal{X}[h](T^*)$ with $\partial_t(\Theta,\theta)\in L^2(S^*\times\Omega)\times L^2(Q[h](T^*))$ satisfy
    \begin{multline*}
    \int_S\int_{\Omega}\partial_t(C(h)\Theta)\Phi\di{x}+\int_\Omega\int_{Y(h)}\partial_t\theta\phi\di{y}\di{x}+\int_{\Omega}K(h)\nabla\Theta\cdot\nabla\Phi\di{x}
    +\int_\Omega\int_{Y(h)}\kappa\nabla_y\theta\cdot\nabla_y\phi\di{y}\di{x}\\
    =\int_{\Omega}\left(F- L(h)\partial_th\right)\Phi\di{x}
    +\int_\Omega\int_{Y(h)}f\phi\di{y}\di{x}
    \end{multline*}
    for all $(\Phi,\phi)\in \mathcal{X}[h]$.
    \item[$(iv)$] The initial conditions are satisfied: $(\Theta(0),\theta(0), h(0))=(\Theta_0,\theta_0,0)$.
\end{itemize}
\end{definition}
This solution concept, however, is inconvenient for the mathematical analyses due to the changing domains.
We use the concept of the Hanzawa transformation to reformulate this problem in the fixed domain $Y_0=Y(0)$.
The precise definition and properties of this transformation are introduced in \cref{subsec:hanzawa} and specifically \cref{hanzawa_transform}.

Now, for any $0<T^*<T$ and height function $h\in W^{1,\infty}(0,T^*;L^\infty(\Omega))$ with $|h|\leq a^*$, let $s_h\in W^{1,\infty}(0,T^*;L^\infty(\Omega;C^2(\overline{Y})))$ be the corresponding Hanwaza transformation.
We introduce the Jacobian matrix $F_h=D_ys_h$, the Jacobian determinant $J_h=\det F_h$, and the velocity $\partial_ts_h$. 
Also, we denote the push-forward $S_h$ and the pull-back $S_h^{-1}$ for any sufficiently regular function $f\colon\overline{Y}\to\R$ defined via $S_hf=f\circ s_h$ and $S_h^{-1}f=f\circ s_h^{-1}$.
Finally, we introduce the transformed quantities 
\begin{alignat*}{2}
    c(h)&=J_h,&\qquad    \kappa(h)&=J_hF_h^{-1}\kappa F_h^{-T},\\
    f(h)&=J_hS_hf,&\qquad v(h)&=J_hF_h^{-1}\partial_ts_h.  
\end{alignat*}
With this transformation, we can alternatively formulate the problem in fixed coordinates where we use $\mathcal{X}_0(T^*):=\mathcal{X}[0](T^*)$:

\begin{definition}[Weak solution in fixed coordinates]\label{def:weak_sol_fixed}
We say that $(h,\Theta,\vartheta)$ is a local-in-time weak solution in moving coordinates, when there is a time horizon $S^*=(0,T^*)$, such that:

\begin{itemize}
    \item[$(i)$] The domain changes are characterized via the height function $h\in W^{1,\infty}(S^*;L^\infty(\Omega))$ which satisfy
    \[
    \|h\|_{L^\infty(S^*\times\Omega)}\leq a^*.
    \]
    \item[$(ii)$] The velocity is driven by the temperature:
    \[
    \partial_th=\Theta-\Theta^{ref}\  \text{in}\ L^2(S^*\times\Omega).
    \]
    \item[$(iii)$] The temperatures $(\Theta,\vartheta)\in \mathcal{X}_0(T^*)$ with $\partial_t(\Theta,\vartheta)\in L^2(S^*\times\Omega)\times L^2(S^*\times\Omega\times Y_0)$ satisfy
    \begin{multline}\label{eq:weak_form_reference}
    \int_S\int_{\Omega}\partial_t(C(h)\Theta)\Phi\di{x}+\int_\Omega\int_{Y_0}\partial_t(c(h)\vartheta)\phi\di{y}\di{x}+\int_{\Omega}K(h)\nabla\Theta\cdot\nabla\Phi\di{x}
    +\int_\Omega\int_{Y_0}\kappa(h)\nabla_y\vartheta\cdot\nabla_y\phi\di{y}\di{x}\\
    +\int_\Omega\int_{Y_0}\vartheta v(h)\cdot\nabla_y\phi\di{y}\di{x}
    =\int_{\Omega}\left(F- L(h)\partial_th\right)\Phi\di{x}
    +\int_\Omega\int_{Y_0}f(h)\phi\di{y}\di{x}
    \end{multline}
    for all $(\Phi,\phi)\in \mathcal{X}_0$.
    \item[$(iv)$] The initial conditions are satisfied: $(\Theta(0),\vartheta(0),h)=(\Theta_0,\vartheta_0,0)$.
\end{itemize}
\end{definition}

It is relatively straightforward to show that $(h,\Theta,\theta)$ is a solution in the sense of \cref{def:weak_sol_moving} if and only if $(h,\Theta,S_h^{-1}\vartheta)$ is a solution in the sense of \cref{def:weak_sol_fixed} (see, e.g., \cite[Lemma 3]{Wiedemann23}).
In \cref{sec:analysis}, we will work with \cref{def:weak_sol_fixed} to show the existence of local-in-time weak solutions.
\section{Analysis of the moving boundary problem}\label{sec:analysis}
In this section, we investigate the two-scale problem with changing micro-domains and show that there are weak solutions $(h,\Theta,\vartheta)$ in the sense of \cref{def:weak_sol_fixed} over some possibly small time interval $(0,T^*)\subset S$.
We also establish uniqueness in case of higher regularity and long-time existence under certain smallness conditions.
Our strategy and outline of this section is the following:

\begin{enumerate}
    \item[$(i)$] In \cref{subsec:hanzawa}, we introduce the Hanzawa transformation (\cref{hanzawa_transform}) and establish some uniform positivity and Lipschitz estimates (\cref{lemma:lipschitz_transformation,lemma:lipschitz_parameters,lemma:lipschitz_cell_functions}).
    \item[$(ii)$] We go on, in \cref{subsec:analysis_prescribed}, to look at the linearized problem by solving the heat problem for a prescribed height function and establishing a-priori estimates (\cref{lemma:wellposedness_given_h}).
    \item[$(iii)$] Finally, in \cref{subsec:analysis_fixed_point}, we employ a Schauder fixed-point argument to prove the existence of weak solutions (\cref{thm:existence}).
    We go on to show uniqueness (\cref{lemma:uniqueness}) and long-time existence (\cref{lemma:long_time}) when certain conditions are met.    
\end{enumerate}

\subsection{Hanzawa transformation and analysis of the cell problems}\label{subsec:hanzawa}
In the following, we present some auxiliary results regarding the Hanzawa transformation which are necessary to estimate and control the impact of the changing geometry on the overall system.
For more details regarding this type of setup, including the case of nonuniform growth, we refer to \cite{Pruss2016}.

As a $C^3$-interface, $\Gamma_0=\partial Y_0$ has a tubular neighborhood $U_\Gamma\subset\subset Y$ with tubular radius $a>0$.
In $U_\Gamma$, every point is uniquely characterized by its projection onto $\Gamma_0$ (via $P\colon U_\Gamma\to\Gamma_0$), and its signed distance (via $d\colon U_\Gamma\to(-a,a)$)\footnote{Here, we follow the convention that negative distances correspond to points inside $Y_0$.}.
Moreover, the function
\[
\Lambda\colon\Gamma_0\times(-a,a)\to U_\Gamma,\quad \Lambda(\gamma,h)=\gamma+hn(\gamma)
\]
is a $C^2$-isomorphism with inverse
\[
\Lambda^{-1}\colon U_\Gamma\to\Gamma_0\times(-a,a),\qquad \Lambda^{-1}(y)=(P(y),d(y)).
\]
We set $a^*=\nicefrac{a}{8}$. 
For $h\in[-a^*,a^*]$, we introduce
\[
\Gamma(h)=\left\{\gamma+h n(\gamma)\ : \ \gamma\in\Gamma_0\right\},\ Y(h)=\inside(\Gamma(h)).
\]
Clearly, $\Gamma(0)=\Gamma_0$, $Y(0)=Y_0$, as well as $\partial Y(h)=\Gamma(h)$.
In the following, let $\chi\colon\R\to[0,1]$ be a smooth cut-off function that satisfies
\begin{alignat*}{2}
(i)&\ \chi(r)=1 \ \text{for}\ r\in\left(-\frac{1}{2},\frac{1}{2}\right), &\qquad (ii)&\ \chi(r)=0 \ \text{for}\ r\notin(-1,1),\\
(iii)&\ \mathrm{sign}(r)\chi'(r)\leq0,&\quad (iv)&\ |\chi'(r)|\leq\frac{5}{2}
\end{alignat*}
and introduce $\psi\colon\overline{U_\Gamma}\to\R^d$ given by $\psi(y)=n(P(y))\chi\left(\nicefrac{d(y)}{a}\right)$.
Note that $\psi\in C^2(\overline{U})$ with $|D_y\psi(y)|<\nicefrac{4}{a}$.\footnote{This follows as the mean curvature is bounded by $a^{-1}$.}
Then, the so-called Hanzawa transform $s_h\colon\overline{Y}\to\overline{Y}$ defined via
\begin{equation}\label{hanzawa_transform}
     s_h(y)=\begin{cases}
     y+h\psi(y)\quad &y\in U_\Gamma,\\
     y\quad &y\in \overline{Y}\setminus U_\Gamma
     \end{cases}
\end{equation}
is a $C^2$-transformation satisfying $s_h(\Gamma)=\Gamma(h)$ and $|D_ys_h|, |D_ys_h^{-1}|\leq 2$ for all $h\in[-a^*,a^*]$.
Please note that due to $h$ being constant this is all a special case of the more general results for the Hanzawa transformation as outlined in \cite[Chapter 2]{Pruss2016}.
In the following, we denote the Jacobian and the Jacobian derivative by $F_h=D_ys_h$ and $J_h=\det F_h$, respectively.

\begin{lemma}\label{lemma:lipschitz_transformation}
    \begin{itemize}
        \item[$(i)$] There are $0<c<C$ independent of $h\in[-a^*,a^*]$ and $y\in \overline{Y}$ such that
        \[
        c\leq J_h(y)\leq C,\quad c|\xi|^2\leq F^{-1}_h(y)F^{-T}_h(y)\xi\cdot\xi\leq C|\xi|^2 \quad(\xi\in\R^d).
        \]
        \item[$(ii)$] Let $h_1,h_2\in [-a^*,a^*]$ and let $s_i=s_{h_i}$, $F_i=F_{h_i}$, $J_i=J_{h_i}$ denote the corresponding transformation functions.
        Then,
        \[
        \|F_1-F_2\|_{L^\infty(Y)^{d\times d}}+\|F_1^{-1}-F_2^{-1}\|_{L^\infty(Y)^{d\times d}}
        +\|J_1-J_2\|_{L^\infty(Y)}\leq C|h_1-h_2|
        \]
        where $C>0$ independent of $h$.
    \end{itemize}
\end{lemma}
\begin{proof}
  $(i)$. These bounds exist as $|F_h(y)|, |F_h^{-1}(y)|\leq 2$ uniformly and independently of $h\in[-a^*,a^*]$.
  
  $(ii)$. Outside the tubular neighborhood $U_\Gamma$, the transformation is trivial, i.e., $F_i=\ID$ and $J_i=1$, so we only have to look at $y\in U$.
  We immediately get
  \[
  |F_1(y)-F_2(y)|=|\left(h_1-h_2\right)D_y\psi(y)|\leq C|h_1-h_2|.
  \]
  We already know that the $F_i$ are invertible with
  \[
  F_i^{-1}(y)=(\ID+h_iD_y\psi(y))^{-1}
  \]
  from where
  \[
  F_i^{-1}(y)=\ID-h_iD_y\psi(y)F_{i}^{-1}(y)
  \]
  follows.
  As a consequence,
  \[
  F_1^{-1}(y)-F_2^{-1}(y)=D_y\psi(y)\left(h_2F_2^{-1}(y)-h_1F_1^{-1}(y)\right)
  \]
  or,
  \[
  |F_1^{-1}(y)-F_2^{-1}(y)|\leq C|h_1-h_2|+|h_2D_y\psi(y)||F_1^{-1}(y)-F_2^{-1}(y)|.
  \]
  Since $|h_2D_y\psi(y)|\leq\nicefrac{1}{2}$, the Lipschitz estimate follows.
  Finally, for the determinant a second order Taylor approximation yields
  \[
  J_i(y)=\det(\ID+h_i D_y\psi(y))=1+h_i\left(\tr(D_y \psi(y))+\mathcal{O}(h_i)\right)
  \]
  showing the Lipschitz estimate.
\end{proof}

In the following, for any $h\in[-a^*,a^*]$, we note that
\[
S_h\colon H^1(Y\setminus Y_0)\to H^1(Y\setminus Y(h)), \qquad S_h^{-1}\colon H^1(Y\setminus Y(h))\to H^1(Y\setminus Y_0)).
\]
Based on the estimates in \cref{lemma:lipschitz_transformation}, there are constants $c,C>0$ such that
\begin{equation}\label{eq:equivalence_norms}
c\|S_h\phi\|_{H^1(Y\setminus Y(h))}\leq\|\phi\|_{H^1(Y\setminus Y_0)}\leq C\|S_h\phi\|_{H^1(Y\setminus Y(h))}\qquad (\phi\in H^1(Y\setminus Y_0)).
\end{equation}

In the next step, we establish how the different coefficients depend on the height parameter and the corresponding transformation.
To that end, let $h\in[-a^*,a^*]$ be given.
It is clear, that $h\mapsto|Y(h)|$ and $h\mapsto|\Gamma(h)|$ are $C^1$ and therefore also Lipschitz. 
We now look at the cell problem given via \cref{eq:system:cell_problem}:

\begin{lemma}\label{lemma:lipschitz_cell_functions}
    For every $h\in[-a^*,a^*]$, there is a unique zero-average solution $\xi_{j,h}\in H_\#(Y\setminus Y(h))$, $j=1,\dots,d$, satisfying    
    \begin{alignat*}{2}
    \int_{Y\setminus Y(h)}\mathcal{K}(\nabla_y\xi_{j,h}+e_j)\cdot\nabla_y\phi\di{y}=0
    \end{alignat*}
    for all $\phi\in H^1_\#(Y\setminus Y(h))$ as well as $\|\nabla\xi_{j,h}\|_{L^2(Y\setminus Y(h))}\leq C$ where $C>0$ is independent of $h$.
    
    Moreover, for $h_1,h_2\in[-a^*,a^*]$, we have
    \[
    \left|\int_{Y\setminus Y(h_1)}\mathcal{K}\nabla\xi_{j,h_1}\cdot e_i\di{y}-\int_{Y\setminus Y(h_2)}\mathcal{K}\nabla\xi_{j,h_2}\cdot e_i\di{y}\right|\leq C |h_1-h_2|.
    \]
\end{lemma}
\begin{proof}
    The existence of a unique zero-average solution is standard and follows via Lax-Milgram \cite[Theorem 2.2]{S96}.
    Let $h\in[-a^*,a^*]$.
    By testing with $\varphi = \xi_{j,h}$ and by assumption (A1) we obtain the desired estimate
    \begin{equation*}
        c \|\nabla\xi_{j,h}\|_{L^2(Y\setminus Y(h))}^2 \leq C \|e_j\|_{L^2(Y\setminus Y(h))} \|\nabla\xi_{j,h}\|_{L^2(Y\setminus Y(h))} = C \underbrace{\sqrt{|Y\setminus Y(h)|}}_{\leq 1}\|\nabla\xi_{j,h}\|_{L^2(Y\setminus Y(h))}. 
    \end{equation*}

    Now, let $h_1,h_2\in[-a^*,a^*]$, take the difference of their respective weak forms, and test with $\bar{\xi}_{j}:=S_1^{-1}\xi_{j,h_1}-S_2^{-1}\xi_{j,h_2}$:
    \begin{multline*}
      \int_{Y\setminus Y_0}J_1F_1^{-1}\mathcal{K}F^{-T}_1\nabla_y\bar{\xi}_{j}\cdot \nabla_y\bar{\xi}_{j}\\
      =\int_{\Gamma_0}\left(J_1-J_2\right)e_j\cdot n\bar{\xi}_{j}\di{y}
      +\int_{Y\setminus Y_0}\left(J_2F_2^{-1}\mathcal{K}F^{-T}_2-J_1F_1^{-1}\mathcal{K}F^{-T}_1\right)\nabla_y(S_2^{-1}\xi_{j,h_2})\cdot\nabla_y\bar{\xi}_{j}\di{y}.
    \end{multline*}
    With the uniform positivity of $J_1F^{-T}_1$ and the Lipschitz continuity of $J_h$ and $F_h^{-T}$ (\cref{lemma:lipschitz_transformation}), we estimate
    \begin{equation*}
    c\|\nabla_y\bar{\xi}_{j}\|^2_{L^2(Y\setminus Y_0)}\\
    \leq C|h_1-h_2|\left(\|\bar{\xi}_{j}\|_{L^1(\Gamma_0)}+\|\nabla_y(S_2^{-1}\xi_{j,h_2})\|_{L^2(Y\setminus Y_0)}\|\nabla_y\bar{\xi}_{j}\|_{L^2(Y\setminus Y_0)}\right).
    \end{equation*}
    Via Poincare's inequality and (\ref{eq:equivalence_norms}), we therefore get
    \begin{equation}\label{eq:lipschitz_xi_h}
    \|\bar{\xi}_{j}\|_{H^1(Y\setminus Y_0)}
    \leq C|h_1-h_2|.
    \end{equation}
    From here, we can conclude the desired Lipschitz estimate via
    \begin{multline*}
    \int_{Y\setminus Y(h_1)}\mathcal{K}\nabla\xi_{j,h_1}\cdot e_i\di{y}
    -\int_{Y\setminus Y(h_2)}\mathcal{K}\nabla\xi_{j,h_2}\cdot e_i\di{y}\\
    =
    \int_{Y\setminus Y_0}\left(J_1F_1^{-1}\mathcal{K} F^{-T}_1\nabla_y(S_1\xi_{j,h_1})-J_2F_2^{-1}\mathcal{K} F^{-T}_2\nabla_y(S_2\xi_{j,h_2})\right)\cdot e_i\di{y}
    \end{multline*}
    and \cref{lemma:lipschitz_transformation}.
\end{proof}

\begin{lemma}\label{lemma:lipschitz_parameters}
\begin{itemize}
    \item[(i)]  Let $h_1,h_2\in[-a^*,a^*]$.
    There is a $C>0$ independent of $h_1,h_2$ such that
    \[
    |C(h_1)-C(h_2)|+|K(h_1)-K(h_2)|+| L(h_1)- L(h_2)|\leq C|h_1-h_2|.
    \]
    \item[(ii)]
    There is $c>0$ such that, for all $h\in[-a^*,a^*]$ and all $\rho\in\R^d$,
    \[
    |C(h)|\geq c,\quad K(h)\rho\cdot\rho\geq c|\rho|^2.
    \]
\end{itemize}
\end{lemma}
\begin{proof}
    $(i)$. This is easy to check for $C$ and $ L$ by directly utilizing the transformation \cref{hanzawa_transform} and the Lipschitz estimates in \cref{lemma:lipschitz_transformation}.

    For the heat conductivity, we rely on additional estimates for the cell problem given in \cref{lemma:lipschitz_cell_functions}:
    Their difference can be estimated by
    \begin{multline*}
    K(h_1)_{i,j}-K(h_2)_{i,j}=\int_{Y\setminus Y(h_1)}\mathcal{K}\nabla_y\xi_{j,h_1}\cdot e_i\di{y}
    -\int_{Y\setminus Y(h_2)}\mathcal{K}\nabla_y\xi_{j,h_2}\cdot e_i\di{y}\\
    +\int_{Y\setminus Y(h_1)}\mathcal{K}_{ij}\di{y}
    -\int_{Y\setminus Y(h_2)}\mathcal{K}_{ij}\di{y}
    \end{multline*}
    where we can apply \cref{lemma:lipschitz_transformation,lemma:lipschitz_cell_functions} to get the Lipschitz estimate.

    $(ii)$. Trivially true for $C(h)=1-|Y(h)|$ as $\dist(\Gamma_0,\partial Y)\geq 2a^*$.
    For the macroscopic conductivity matrix, positivity can be shown easily for a fixed $h$ with standard techniques (e.g., \cite[Theorem 12.5]{stuartpavliotis}, i.e., $K(h)\rho\cdot\rho\geq c_h|\rho|^2$.
    Since $K$ is continuous in $h$ (shown in $(i)$) and $[-a^*,a^*]$ is compact, uniform positivity follows with $c=\min_{|h|\leq a^*}c_h>0$.    
\end{proof}

\subsection{Analysis for prescribed height function}\label{subsec:analysis_prescribed}
We take a look at the weak form for the heat system, i.e., \cref{eq:weak_form_reference}, for a prescribed height function $h$ and establish well-posedness and some further auxiliary results for this linearized problem.
For some arbitrary and later to be fixed $M>0$, we introduce the space of admissible height functions whose time derivatives are bounded by $M$:
\[
\mathcal{H}_M(t)=\{h\in W^{1,\infty}((0,t); L^\infty(\Omega))\ : \ h(0,x)=0,\ |\partial_th(\tau,x)|\leq M \ \text{for a.a. $\tau\in(0,t)$ and $x\in\Omega$}\}.
\]
Please note that every $h\in\mathcal{H}_M(t)$ satisfies $|h|\leq a^*$ as long as $t\leq \nicefrac{a^*}{M}=:T_M$; therefore the results from \cref{subsec:hanzawa} apply pointwise almost everywhere.
We set $S_M=(0,T_M)$.
\begin{lemma}[Existence for prescribed height functions]\label{lemma:wellposedness_given_h}
Let $M>0$ and $h\in \mathcal{H}_M(T_M)$.
There exists a unique weak solution $(\Theta,\vartheta)\in \mathcal{X}_0$ with $\partial_t(\Theta,\vartheta)\in L^2(S_M\times\Omega)\times L^2(S_M\times\Omega\times Y_0)$ satisfying
\begin{multline*}
    \int_{\Omega}\partial_t(C(h)\Theta)\Phi\di{x}+\int_\Omega\int_{Y_0}\partial_t(c(h)\vartheta)\phi\di{y}\di{x}+\int_{\Omega}K(h)\nabla\Theta\cdot\nabla\Phi\di{x}
    +\int_\Omega\int_{Y_0}\kappa(h)\nabla_y\vartheta\cdot\nabla_y\phi\di{y}\di{x}\\
    +\int_\Omega\int_{Y_0}\vartheta v(h)\cdot\nabla\phi\di{y}\di{x}
    =\int_{\Omega}\left(F- L(h)\partial_th\right)\Phi\di{x}
    +\int_\Omega\int_{Y_0}f(h)\phi\di{y}\di{x}
\end{multline*}
for all $(\Phi,\phi)\in \mathcal{X}_0$ and almost all $t\in S_M$.
In addition, there is $C>0$, which is independent of $h$ and $M$, such that
\begin{equation*}
\|\Theta\|_{L^\infty(S_M;L^2(\Omega))}+\|\Theta\|_{L^2(S_M;H^1(\Omega))}+\|\vartheta\|_{L^\infty(S_M;L^2(\Omega\times Y_0))}+\|\vartheta\|_{L^2(S_M\times\Omega;H^1(Y_0))}
\leq C\exp\left(CM\right)
\end{equation*}
Moreover, there is $C>0$ such that for all $h\in\mathcal{H}_M(T_M)$:
\[
\|\Theta\|_{L^\infty((0,t)\times\Omega)} + \|\vartheta\|_{L^\infty((0,t)\times\Omega\times Y_0)}\leq C\left(1+M^2\sqrt{t}\right)
\]
\end{lemma}
\begin{proof}
    For any such height function, all coefficients are well-defined and essentially bounded.
    Moreover, $c(h)$ and $C(h)$ are positive and $\kappa(h)$ and $K(h)$ are positive definite (see \cref{lemma:lipschitz_transformation,lemma:lipschitz_parameters}).
    As a consequence, this is a linear implicit evolution equation in the form of
    \[
    \ddt(\mathcal{B}(t)u(t))+\mathcal{A}(t)u(t)=f(t)
    \]
    which under the given regularity admits a unique solution.
    We refer to \cite[Chapter III, Propositions 3.2 \& 3.3]{S96}.

    Regarding the estimates: First of all, the advective velocity $v(h)=J_hF_h^{-1}\partial_ts_h$ and the time derivatives of the coefficients $c(h)$ and $C(h)$ can be estimated by
    \[
    |v(h)|+|\partial_tc(h)|+|\partial_tC(h)|\leq C|\partial_th|\leq CM.
    \]
    Now, testing with $(\Phi,\phi)=(\Theta,\vartheta)$ shows:
    \begin{multline*}
    \ddt\int_{\Omega}\frac{1}{2}C(h)|\Theta|^2\di{x}
    +\ddt\int_{\Omega}\int_{Y_0}\frac{1}{2}c(h)|\vartheta|^2\di{x}
    +c\|\nabla\Theta\|^2_{L^2(\Omega)}
    +c\|\nabla_y\vartheta\|^2_{L^2(\Omega\times Y_0)}\\
    \leq C\int_\Omega M\left(|\Theta|+|\Theta|^2+\int_{Y_0}|\vartheta|^2+|\vartheta\nabla_y\vartheta|\di{y}\right)\di{x}
    +\|F\|_{L^2(\Omega)}\|\Theta\|_{L^2(\Omega)}
    +\|f(h)\|_{L^2(\Omega\times Y_0)}\|\vartheta\|_{L^2(\Omega\times Y_0)}
    \end{multline*}
    for almost all $t\in S_M$.
    Integrating over $(0,t)$ for $t\in S_M$ and using standard energy techniques, we get 
    \begin{multline*}
    \|\Theta(t)\|_{L^2(\Omega)}^2
    +\|\vartheta(t)\|_{L^2(\Omega\times Y_0)}^2
    +\int_0^t\|\nabla\Theta\|^2_{L^2(\Omega)}
    +\|\nabla_y\vartheta\|^2_{L^2(\Omega\times Y_0)}\di{\tau}
    \\
    \leq C\int_0^t\int_\Omega M\left(|\Theta|+|\Theta|^2+(1+M)\int_{Y_0}|\vartheta|^2\di{y}\right)\di{x}\di{\tau}
    \\
    +\int_0^t\|F\|_{L^2(\Omega)}\|\Theta\|_{L^2(\Omega)}
    +\|f(h)\|_{L^2(\Omega\times Y_0)}\|\vartheta\|_{L^2(\Omega\times Y_0)}\di{\tau}.
    \end{multline*}
    On the first term on the right hand side, we can further estimate:
    \[
    \int_0^t\int_\Omega|\partial_th|\left(|\Theta|+|\Theta|^2\right)\di{x}\di{\tau}\leq
    C\int_0^tM\left(\|\Theta\|^2_{L^2(\Omega)}+ 1 \right)\di{\tau}
    \]
    As a result, we arrive at
    \begin{multline*}
    \|\Theta(t)\|_{L^2(\Omega)}^2
    +\|\vartheta(t)\|_{L^2(\Omega\times Y_0)}^2
    +\int_0^t\|\nabla\Theta\|^2_{L^2(\Omega)}
    +\|\nabla_y\vartheta\|^2_{L^2(\Omega\times Y_0)}\di{\tau}
    \\
    \leq C\left(1+ \int_0^t\left(1+M+M^2\right)\left(\|\Theta\|^2_{L^2(\Omega)}+\|\vartheta\|_{L^2(\Omega\times Y_0)}^2\right)\di{\tau}\right).
    \end{multline*}
    With Grönwalls inequality, we are led to 
    \begin{multline*}
    \|\Theta\|_{L^\infty(0,t;L^2(\Omega))}^2
    +\|\vartheta\|_{L^\infty(0,t;L^2(\Omega\times Y_0))}^2
    +\|\nabla\Theta\|^2_{L^2(0,t;L^2(\Omega))}
    +\|\nabla_y\vartheta\|^2_{L^2(0,t;L^2(\Omega\times Y_0))}\\
    \leq C\exp\left(CtM^2\right).
    \end{multline*}
    %

    Regarding the boundedness of the solution, for $k\in\N$ with $k\geq\max\{\|\Theta_0\|_{L^\infty(\Omega)},\|\theta_0\|_{L^\infty(\Omega\times Y_0)}\}$, we introduce the functions $\Theta^k\in L^2(S_M;H^1(\Omega))$ and $\vartheta^k\in L^2(S_M\times\Omega;H^1(Y_0))$ via
    \[
    \Theta^k=\begin{cases}
            \Theta-k &\text{for}\ \Theta> k\\
            0&\text{for}\ |\Theta|\leq k\\
            \Theta+k &\text{for}\ \Theta< -k
            \end{cases},\qquad    \vartheta^k=\begin{cases}
            \vartheta-k &\text{for}\ \vartheta> k\\
            0&\text{for}\ |\vartheta|\leq k\\
            \vartheta+k &\text{for}\ \vartheta< -k
            \end{cases},\quad
    \]
    as well as the sets
    \begin{align*}
    A_\Theta(t,k)&=\{x\in \Omega\ : \ |\Theta(\tau,x)|\geq k\},\\ A_\vartheta(t,x,k)&=\{y\in  Y_0\ : \ |\vartheta(\tau,x,y)|\geq k\}.
    \end{align*}
    Please note that $\partial_t\Theta^k=\chi_{A_\Theta}\partial_t\Theta$, $\nabla\Theta^k=\chi_{A_\Theta}\nabla\Theta$, and $\Theta\Theta^k\geq|\Theta^k|^2$ (similar for $\vartheta$).
    With $(\Theta^k,\vartheta^k)$ as a test function and the \cref{lemma:lipschitz_transformation,lemma:lipschitz_parameters}, standard energy estimates lead to ($t\in S_M$)
    \begin{multline*}
    \|\Theta^k(t)\|_{L^2(\Omega)}^2+\|\vartheta^k(t)\|^2_{L^2(\Omega\times Y_0)}
    +\|\nabla\Theta^k\|^2_{L^2((0,t)\times\Omega)}
    +\|\nabla_y\vartheta^k\|^2_{L^2((0,t)\times\Omega\times Y_0)}\\
    \leq C\int_0^t\int_\Omega\left(|\partial_th||\Theta||\Theta^k|+\left|F- L(h)\partial_th\right||\Theta^k|\right)\di{x}\di{\tau}\\
    +C\int_0^t\int_\Omega\int_{Y_0}|\partial_th||\vartheta||\vartheta^k|+|\vartheta| |v(h)||\nabla_y\vartheta^k|
    +|f(h)||\vartheta^k|\di{y}\di{x}\di{\tau}
    \end{multline*}
    where $C>0$ does not depend on $h$, $k$, or $M$.
    Using $|\Theta|\leq k+|\Theta^k|$, $|\vartheta|\leq k+|\vartheta^k|$, and $|\partial_th|\leq M$, we estimate
    \begin{align*}
    \int_\Omega|\partial_th||\Theta||\Theta^k|\di{x}&\leq 2M\left(\|\Theta^k\|_{L^2(\Omega)}^2+k^2|A_\Theta(t,k)|\right),\\
    \int_{Y_0}|\partial_th||\vartheta||\vartheta^k|\di{y}&\leq 2M\left(\|\vartheta^k\|_{L^2(Y_0)}^2+k^2|A_\vartheta(t,x,k)|\right).
    \end{align*}
    The latent heat term, we estimate via (with $k\geq\max_{h\in[-a^*,a^*]}| L(h)|$)
    \[
    \int_\Omega| L(h)\partial_th||\Theta^k|\di{x}\leq \frac{M}{2}\left(\|\Theta^k\|_{L^2(\Omega)}^2+k^2|A_\Theta(t,k)|\right).
    \]
    Looking at the convection term ($v(h)\leq CM$):
    \[
    \int_{Y_0}|\vartheta| |v(h)||\nabla_y\vartheta^k|\di{y}
    \leq \delta\|\nabla_y\vartheta^k\|_{L^2(Y_0)}^2+\frac{M^2}{4\delta} \left(\|\vartheta^k\|_{L^2(Y_0)}^2+k^2|A_\vartheta(\tau,x,k)|\right).
    \]
    for any $\delta>0$.
    We set
    \[
    k^*=\max\left\{\|\Theta_0\|_{L^\infty(\Omega)},\|\theta_0\|_{L^\infty(\Omega\times Y_0)},\|F\|_{L^\infty((0,T_M)\times\Omega)}, \|f\|_{L^\infty((0,T_M)\times\Omega\times Y_0)},\max_{h\in[-a,a]}| L(h)|\right\}
    \]
    and conclude, for all $k\geq k^*$ and assuming $M\geq1$,
    \begin{multline*}
    \|\Theta^k(t)\|^2_{L^2(\Omega)}+\|\vartheta^k(t)\|^2_{L^2(\Omega\times Y_0)}
    +\|\nabla\Theta^k\|^2_{L^2((0,t)\times\Omega)}
    +\|\nabla_y\vartheta^k\|^2_{L^2((0,t)\times\Omega\times Y_0)}\\
    \leq CM^2\Bigg(\|\theta^k\|_{L^2((0,t)\times\Omega)}^2+    
    \|\vartheta^k\|_{L^2((0,t)\times\Omega\times Y_0)}^2
    +k^2\int_0^t\left(|A_\theta(\tau,k)|+\int_\Omega|A_\vartheta(\tau,x,k)|\di{x}\right)\di{\tau}\Bigg).
    \end{multline*}
    Thus, Gronwall's inequality allows us to conclude for $0<t<T_M$
    \begin{multline*}
    \|\Theta^k\|_{L^\infty(0,t;L^2(\Omega))}^2+\|\vartheta^k\|^2_{L^\infty(0,t;L^2(\Omega\times Y_0))}
    +\|\nabla\Theta^k\|_{L^2(0,t\times\Omega)}^2
    +\|\nabla_y\vartheta^k\|^2_{L^2(0,t\times\Omega\times Y_0)}\\
    \leq C\exp(tM^2)M^2k^2\int_{0}^t\left(|A_\theta(\tau,k)|+\int_\Omega|A_\vartheta(\tau,x,k)|\di{x}\right)\di{\tau}.
    \end{multline*}
    where the constant $C>0$ is still independent of $k$, $h$, and $M$.
    From here, we apply \cite[Chapter 2, §6, Theorem 6.1, Remark 6.2]{Ladyzenskaja} to conclude the existence of $k_M$, which is independent of $h$ but does depend on $M$, such that $|A_\theta(\tau,k_M)|,|A_\vartheta(\tau,x,k_M)|=0$ almost everywhere.
    More precisely
    \[
    \|\Theta\|_{L^\infty((0,t)\times\Omega)}+\|\vartheta\|_{L^\infty((0,t)\times\Omega)}\leq 2k^*\left(1+C\exp(\nicefrac{5}{4}M^2t)M^2 \sqrt{t} \right)\quad (t\in(0,T_M))
    \]
    Moreover, $tM^2\leq a^*$ for $0\leq t\leq\sqrt{T_M}$.
    Over this time interval, we have
    \[
    \|\Theta\|_{L^\infty((0,t)\times\Omega)}+\|\vartheta\|_{L^\infty((0,t)\times\Omega)}\leq 2k^*\left(1+CM^2\sqrt{t}\right) \quad (t\in (0,\sqrt{T_M}).
    \]
    This implies uniformity in $M$ of the $L^\infty$ estimates for $t\leq \nicefrac{a^*}{M^4}$.
\end{proof}

\subsection{Fixed-point argument and long time behavior}\label{subsec:analysis_fixed_point}
We first note that \cref{lemma:wellposedness_given_h} induces a solution map $\mathcal{L}:=(\mathcal{L}_1(h),\mathcal{L}_2(h))\colon\mathcal{H}_M(T_M)\to\mathcal{X}_0(T_M)$ for every $M>0$.
Here, $\mathcal{L}_1(h)\in L^2(S_M\times\Omega)$ denotes the macroscopic temperature and $\mathcal{L}_2(h)\in L^2(S_M\times\Omega\times Y_0)$ the microscopic temperature, respectively.
Also, with \cref{lemma:wellposedness_given_h}, it holds
\[
C_L:=\sup_{h\in\mathcal{H}_M(t)}\|\mathcal{L}_1(h)\|_{L^\infty((0,t)\times\Omega)}<\infty
\]
where, importantly, $C_L$ does not depend on $M$ as long as $t\in(0,\nicefrac{a^*}{M^4})$.
Setting
\[
\mathcal{Y}(t):=\left\{(\Theta,\vartheta)\in\mathcal{X}_0(t)\ : \ \|\Theta\|_{L^\infty((0,t)\times\Omega)}\leq C_L\right\}\subset\mathcal{X}(t),
\]
and $T^*_M= \nicefrac{a^*}{M^{-4}}$, it is clear that $\mathcal{L}[\mathcal{H}_M(T_M^*)]\subset \mathcal{Y}(T_M^*)$.
Let $(\Theta,\vartheta)\in \mathcal{Y}(T_M^*)$.
To get to a fixed-point argument, we must be able to construct a new height function $\tilde{h}$ based on the macroscopic temperature $\Theta$.
We introduce $\tilde{h}(t,x)=\int_0^t\Theta(\tau,x)-\Theta^{ref}\di{\tau}$ and note that $\tilde{h}\in W^{1,\infty}((0,T_M^*);L^\infty(\Omega))$ via the regularity of $\Theta$.
Moreover, %
\[
|\partial_t\tilde{h}(t,x)|=|\Theta(t,x)-\Theta^{ref}|\leq C_L+|\Theta^{ref}|
\]
for almost all $(t,x)\in (0,T_M^*)\times\Omega$.
As a consequence, $\tilde{h}\in\mathcal{H}_M(T_M)$ as long as $M\geq C_L+|\Theta^{ref}|$.
We fix $M^*=C_l+|\Theta^{ref}|$, set $T^*:=T_{M}^*$ and $S^*=(0,T_M^*)$, and introduce
\[
\mathcal{T}\colon \mathcal{Y}(T^*)\to \mathcal{Y}(T^*),\quad \mathcal{T}(\Theta,\vartheta)=\mathcal{L}\left(\int_0^t\Theta(\tau,x)-\Theta^{ref}\di{\tau}\right).
\]

\begin{theorem}[Existence of weak solutions]\label{thm:existence}
    There is at least one solution $(h,\Theta,\vartheta)\in\mathcal{H}_{M^*}(T^*)\times\mathcal{X}_0(T^*)$ for the moving boundary problem in fixed coordinates in the sense of \cref{def:weak_sol_fixed} over the time interval $S^*=(0,T^*)$.
\end{theorem}
\begin{proof}
We wirst note, that $(h,\Theta,\vartheta)\in\mathcal{H}_{M^*}(T^*)\times\mathcal{X}_0(T^*)$ is a solution if, and only if, $\mathcal{T}(\Theta,\vartheta)=(\Theta,\vartheta)$ and $h(t,x)=\int_0^t\Theta(\tau,x)-\Theta^{ref}\di{\tau}$.

\textit{(i) Continuity of fixed-point operator over $L^2$.} Let $(\Theta_n,\vartheta_n),(\Theta_\infty,\vartheta_\infty)\in\mathcal{Y}(T^*)$ such that $(\Theta_n,\vartheta_n)\to(\Theta_\infty,\vartheta_\infty)$ in $L^2(S^*\times\Omega)\times L^2(S^*\times\Omega\times Y_0)$.
Let $h_n,h_\infty\in\mathcal{H}_{M^*}(T^*)$ denote the corresponding height functions.
Due to $\partial_th_k=\Theta_k$ ($k\in\N_\infty)$, we immediately have $\partial_th_n\to \partial_t h_\infty$ and $h_n\to h_\infty$ in $L^2(S^*\times\Omega)$.
Now, we set $(\widehat{\Theta}_k,\widehat{\vartheta}_k):=\mathcal{T}(\Theta_k,\vartheta_k)\in\mathcal{Y}(T^*)$ for $k\in\N_\infty$.
Utilizing the a priori estimates established with \cref{lemma:wellposedness_given_h}, there exist $(\Theta,\vartheta)\in \mathcal{Y}(T^*)$ such that, at least along a subsequence, $(\widehat{\Theta}_n,\widehat{\vartheta}_n)\to (\widehat{\Theta},\widehat{\vartheta})$ weakly in $\mathcal{X}_0(T^*)$ and, with Lions-Aubin, strongly in $L^2$.
If we can show that $(\widehat{\Theta},\widehat{\vartheta})=(\widehat{\Theta}_\infty,\widehat{\vartheta}_\infty)$ and that the full sequence converges, continuity holds.

Looking at the most critical terms in the weak forms, namely the heat conductivities, we have the limits
\begin{align*}
    \lim_{j\to\infty}\int_{\Omega}K(h_{n_j})\nabla\widehat{\Theta}_{n_j}\cdot\nabla\Phi\di{x}
    &=\int_{\Omega}K(h_\infty)\nabla \widehat{\Theta}\cdot\nabla\Phi\di{x},\\
    \lim_{j\to\infty}\int_\Omega\int_{Y_0}\kappa(h_{n_j})\nabla_y\vartheta_{n_j}\cdot\nabla_y\phi\di{y}\di{x}&=\int_\Omega\int_{Y_0}\kappa(h)\nabla_y\widehat{\vartheta}\cdot\nabla_y\phi\di{y}
\end{align*}
for all $(\Phi,\varphi)\in \mathcal{X}_0(T^*)$.
Here, we used the weak convergences of $(\nabla\widehat{\Theta}_n,\nabla_y\widehat{\vartheta_{n}})$ along a subsequence ${n_j}$ and the strong convergence of $h_k$ together with the uniform bounds for $K,\kappa$.
After identifying all terms in the weak form in the same way, we can see that $\widehat{\Theta}=\widehat{\Theta}_\infty$, which implies $\widehat{\Theta}_{n_j}\to\widehat{\Theta}_\infty$ strongly in $L^2(S^*\times\Omega)$, similarly for $\hat{\vartheta}_{n_j}$.
Finally, since every subsequence $(\widehat{\Theta}_n,\widehat{\vartheta}_n)$ has a converging subsequence which all converge to the same limit, the whole sequence converges implying $L^2$-continuity of $\mathcal{T}$.

\textit{(ii) Existence of solutions.} We showed that
\[\mathcal{T}\colon \mathcal{Y}(T^*)\subset L^2(S^*\times\Omega)\times L^2(S^*\times\Omega\times Y_0)\to L^2(S^*\times\Omega)\times L^2(S^*\times\Omega\times Y_0)
\]
is continuous.
Moreover $\mathcal{T}[\mathcal{Y}(T^*)]$ is compact in  $L^2(S^*\times\Omega\times Y_0)$ via Lions-Aubin and $\mathcal{Y}(T^*)$ is closed and convex.
With Schauders fixed point theorem, there must therefore be at least one fixed point $(\Theta,\vartheta)\in \mathcal{Y}(T^*)$ satisfying $\mathcal{T}(\Theta,\vartheta)=(\Theta,\vartheta)$.
\end{proof}
\begin{remark}\label{remark:analysis_initial_height}
    Please note that all the arguments hold true when starting with a non-trivial height function, i.e., when $h(0)=h_0$ for some $h_0\in L^\infty(\Omega)$ with $c_0:=\|h_0\|_{L^\infty(\Omega)}<a^*$.
    It does however lead to a shorter time horizon, as now $M^*=c_0+C_L+|\Theta^{ref}|$ and $T^*=\nicefrac{(a^*-c_0)}{(M^*)^4}$.
\end{remark}

This argument does not guarantee uniqueness, however, this follows under slightly higher regularity via Lipschitz arguments (see \cref{lemma:uniqueness} below).
\Cref{thm:existence} ensures the existence of solutions for the moving boundary problem in fixed-coordinates, but using the push forward $S_h$ the corresponding solutions in moving coordinates are given via $(\Theta,S_h\vartheta, h)$.

\begin{lemma}[Uniqueness]\label{lemma:uniqueness}
If $\mathcal{T}$ has a fixed point $(\Theta,\vartheta)\in\mathcal{Y}(T^*)$ with the additional regularity $\partial_t\Theta, \nabla\Theta\in L^\infty((0,T^*)\times\Omega)$ and $\partial_t\vartheta, \nabla_y\vartheta\in L^\infty((0,T^*)\times\Omega;L^2(Y_0))$.
Then, $(\Theta,\vartheta)$ is the only fixed point.
\end{lemma}
\begin{proof}
    Let $(\Theta_i,\vartheta_i)\in\mathcal{Y}(T^*)$ ($i=1,2$) be fixed-points of $\mathcal{T}$ with the higher regularity for $i=2$ with corresponding height functions $h_i\in\mathcal{H}_{M^*}(T^*)$.
    Taking the difference of their weak forms and testing with $(\bar{\Theta},\bar{\vartheta}):=(\Theta_1-\Theta_2,\vartheta_1-\vartheta_2)$:
    \begin{multline*}
    \underbrace{\int_{\Omega}\partial_t(C(h_1)\Theta_1-C(h_2)\Theta_2)\bar{\Theta}\di{x}}_{I_1}+\underbrace{\int_\Omega\int_{Y_0}\partial_t(c(h_1)\vartheta_1-c(h_2)\vartheta_2)\bar{\vartheta}\di{y}\di{x}}_{I_2}\\
    +\underbrace{\int_{\Omega}\left(K(h_1)\nabla\Theta_1-K(h_2)\nabla\Theta_2\right)\cdot\nabla\bar{\Theta}\di{x}}_{I_3}
    +\underbrace{\int_\Omega\int_{Y_0}\left(\kappa(h_1)\nabla_y\vartheta_1-\kappa(h_2)\nabla_y\vartheta_2\right)\cdot\nabla_y\bar{\vartheta}\di{y}\di{x}}_{I_4}\\
    +\underbrace{\int_\Omega\int_{Y_0}(\vartheta_1 v(h_1)-\vartheta_2 v(h_2))\cdot\nabla_y\bar{\vartheta}\di{y}\di{x}}_{I_5}\\
    =\underbrace{\int_{\Omega}( L(h_2)\partial_th_2- L(h_1)\partial_th_1)\bar{\Theta}\di{x}}_{I_6}
    +\underbrace{\int_\Omega\int_{Y_0}(f(h_1)-f(h_2))\bar{\vartheta}\di{y}\di{x}}_{I_7}.
    \end{multline*}
    Using the estimates from \cref{lemma:lipschitz_transformation,lemma:lipschitz_cell_functions,lemma:lipschitz_parameters}, we can now estimate the individual terms.
    Starting with the macroscopic time derivative, we get
    \begin{align*}
    I_1&=
    \int_{\Omega}\partial_t(C(h_1)\bar{\Theta})\bar{\Theta}\di{x}+\int_{\Omega}\partial_t[(C(h_1)-C(h_2))\Theta_2]\bar{\Theta}\di{x}\\
    &=\frac{1}{2}\ddt\int_\Omega C(h_1)\bar{\Theta}^2\di{x}
    -\frac{1}{2}\int_\Omega\partial_tC(h_1)\bar{\Theta}^2\di{x}
    +\int_\Omega\partial_t[(C(h_1)-C(h_2))\Theta_2]\bar{\Theta}\di{x}.
    \end{align*}
    The second and third term on the right hand side satisfy
    \[
    \left|\frac{1}{2}\int_\Omega\partial_tC(h_1)\bar{\Theta}^2\di{x}\right|
    +\left|\int_\Omega\partial_t[(C(h_1)-C(h_2))\Theta_2]\bar{\Theta}\di{x}\right|
    \leq C\|\bar{\Theta}\|_{L^2(\Omega)}\left(\|\bar{\Theta}\|^2_{L^2(\Omega)}+\|\partial_t\bar{h}\|_{L^2(\Omega)}
    +\|\bar{h}\|_{L^2(\Omega)}\right).
    \]
    Similarly, we get for the microscopic time derivative:
    \[
    I_2\geq\frac{1}{2}\ddt\int_\Omega\int_{Y_0}c(h_1)\bar{\vartheta}^2\di{y}\di{x}
    -C\|\bar{\vartheta}\|_{L^2(\Omega\times Y_0)}\left(\|\bar{\vartheta}\|_{L^2(\Omega\times Y_0)}+\|\partial_t\bar{h}\|_{L^2(\Omega)}+\|\bar{h}\|_{L^2(\Omega)}\right)
    \]
    Looking at the macroscopic heat conductivity term:
    \begin{align*}    
    I_3&=\int_{\Omega}K(h_1)\nabla\bar{\Theta}\cdot\nabla\bar{\Theta}\di{x}+\int_{\Omega}\left(K(h_1)-K(h_2)\right)\nabla\Theta_2\cdot\nabla\bar{\Theta}\di{x},\\
    &\geq c\|\nabla\bar{\Theta}\|_{L^2(\Omega)}^2-C\|\bar{h}\|_{L^2(\Omega)}\|\nabla\bar{\Theta}\|_{L^2(\Omega)}
    \end{align*}
    The same type of estimate also holds for the microscopic heat conductivity term $I_4:$
    \begin{align*}
    I_4&=\int_\Omega\int_{Y_0}\kappa(h_1)\nabla_y\bar{\vartheta}\cdot\nabla_y\bar{\vartheta}\di{y}\di{x}
    +\int_\Omega\int_{Y_0}\left(\kappa(h_1)-\kappa(h_2)\right)\nabla_y\vartheta_2\cdot\nabla_y\bar{\vartheta}\di{y}\di{x}\\
    &\geq c\|\nabla_y\bar{\vartheta}\|_{L^2(\Omega\times Y_0)}^2
    -C\|\bar{h}\|_{L^2(\Omega)}\|\nabla_y\bar{\vartheta}\|_{L^2(\Omega\times Y_0)}
    \end{align*}
    The convective part on the microscale can be estimated as
    \begin{align*}
        |I_5|&\leq C(\|\bar{\vartheta}\|_{L^2(\Omega\times Y_0)}\|\nabla_y\bar{\vartheta}\|_{L^2(\Omega\times Y_0)}+\|\bar{h}\|_{L^2(\Omega)}\|\nabla_y\bar{\vartheta}\|_{L^2(\Omega\times Y_0)}),
    \end{align*}
    and the macroscopic heat production via
    \begin{align*}
    |I_6|&\leq C\|\bar{\Theta}\|_{L^2(\Omega)}\left(\|\bar{h}\|_{L^2(\Omega)}
    +\|\partial_t\bar{h}\|_{L^2(\Omega)}\right)
    \end{align*}
    For $I_7$, we use the uniform Lipschitz continuity of $f$ to estimate
    \[
    |I_7|\leq C\|\bar{h}\|_{L^2(\Omega)}\|\bar{\vartheta}\|_{L^2(\Omega;L^1(Y))}.
    \]
    Summarizing and sorting these individual estimates while subsuming the gradient differences via Young’s inequality, we get
    \begin{multline*}
        \ddt\int_\Omega C(h_1)\bar{\Theta}^2\di{x}+\ddt\int_\Omega\int_{Y_0}c(h_1)\bar{\vartheta}^2\di{y}\di{x}
        +\|\nabla\bar{\theta}\|_{L^2(\Omega}^2+\|\nabla_y\bar{\vartheta}\|_{L^2(\Omega\times Y_0)}^2\\
        \leq C\left(\|\bar{\Theta}\|^2_{L^2(\Omega)}+\|\bar{\vartheta}\|^2_{L^2(\Omega\times Y_0)}+\|\bar{h}\|^2_{L^2(\Omega)}+\|\partial_t\bar{h}\|^2_{L^2(\Omega)}\right).
    \end{multline*}
    Integrating over any time interval $(0,t)$ and keeping in mind the uniform positivity of $C(h_1)$ and $c(h_1)$, we are led to
    \begin{multline}\label{eq:unique_gronwall}
        \|\bar{\Theta}(t)\|^2_{L^2(\Omega)}+\|\bar{\vartheta}(t)\|_{L^2(\Omega\times Y_0)}
        +\|\nabla\bar{\theta}\|_{L^2((0,t)\times\Omega}^2+\|\nabla_y\bar{\vartheta}\|_{L^2((0,t)\times\Omega\times Y_0)}^2\\
        \leq C\int_0^t\left(\|\bar{\Theta}\|^2_{L^2(\Omega)}+\|\bar{\vartheta}\|^2_{L^2(\Omega\times Y_0)}+\|\bar{h}\|^2_{L^2(\Omega)}+\|\partial_t\bar{h}\|^2_{L^2(\Omega)}\right)\di{\tau}.
    \end{multline}
    Now, looking at the terms with the height function, we first note that $\partial_t\bar{h}=\Theta$.
    Similarly, we estimate 
    \[
    \int_0^t\|\bar{h}(\tau)\|^2_{L^2(\Omega)}\di{\tau}
    \leq\int_0^t\left(\int_\Omega\int_0^\tau|\bar{\Theta}(s,x)|^2\di{s}\di{x}\right)^{2}\di{\tau}
    =t\int_0^t\|\bar{\Theta}\|^2_{L^2((0,\tau)\times\Omega)}\di{\tau}
    \leq t^2\|\bar{\Theta}\|^2_{L^2((0,t)\times\Omega)}.
    \]
    Applying Grönwall's inequality to \cref{eq:unique_gronwall}, we conclude $\bar{\Theta}=0$, $\bar{\vartheta}=0$ which also implies $\bar{h}=0$.    
\end{proof}

The solution given by \cref{thm:existence} is only local-in-time, which is expected as the Hanzawa transformation can only ever track changes inside the tubular neighborhoods.
In the following lemma, we show that there are no other factors limiting the time horizon. 
This also means that in some very specific scenarios, i.e., whenever the data is very well prepared to ensure that the growth/shrinkage of the cells is controlled, global existence holds.
To that end, we introduce the supremal time horizon
\[
T_{sup}:=\sup\{t\in\R \ : \ \text{there exists a solution}\ (\Theta,\vartheta,h)\ \text{in the sense of \cref{def:weak_sol_fixed}}\}.
\]
\begin{lemma}\label{lemma:long_time}
    Let $T\in\R$, let $(\Theta,\vartheta,h)\in X_0(T)\times \mathcal{H}_{M^*}(T)$ be a solution in the sense of \cref{def:weak_sol_fixed}, and let $\|h\|_{L^\infty((0,T)\times\Omega)}<a^*$.
    Then, the solution can be extended, i.e., $T_{sup}>T$.

    Moreover, if $\|h\|_{L^\infty((0,T')\times\Omega)}<a^*$ for all $0<T'<T_{sup}$, then the solution can be extended globally, i.e, $T_{sup}=\infty$.
\end{lemma}
\begin{proof}
Let $T\in\R$ and let $(\Theta,\vartheta,h)\in \mathcal{X}_0(T)\times \mathcal{H}_{M^*}(T)$ be a solution with $\|h\|_{L^\infty((0,T)\times\Omega)}=a^*-\e$ for some $\e>0$.
Now, taking $Y(a^*-\e)$ and $\Gamma(a^*-\e)$ as the initial geometry and $(\Theta(T),\vartheta(T))$ as initial conditions, all results from this section can be applied.
In this case, we do not have a uniform initial geometry, i.e., the height function is not uniformly 0, but this does not pose problems (see \cref{remark:initial_height,remark:analysis_initial_height}).
Therefore, there is a time horizon $T_\e>0$ and a solution $(\Theta',\vartheta',h')$ on $(0,T_\e)$. 
Gluing $(\Theta',\vartheta',h')$ to $(\Theta,\vartheta,h)$ at time $T$ results in a solution on $(0,T+T_\e)$.

For the second proposition, assume that $\|h\|_{L^\infty((0,T')\times\Omega)}<a^*$ for all $T'<T_{sup}$.
There are two cases:\footnote{The limit must exist as the norm is monotonically increasing and bounded.}
\[
(i) \lim_{T'\to T_{sup}}\|h\|_{L^\infty((0,T')\times\Omega)}=a-\e<a^*,\quad (ii) \lim_{T'\to T_{sup}}\|h\|_{L^\infty((0,T')\times\Omega)}=a^*.
\]
In the first case, $T_{sup}=\infty$ because otherwise we could extend via the first part (choose $T'<T_{sup}$ such that $T_{sup}<T'+T_\e)$.
In the second case, if $T_{sup}$ is finite the solution can not be extended to $T_{sup}$ (since we assumed $\|h\|_{L^\infty((0,T')\times\Omega)}<a^*$).

To show, on the contrary, that it can be extended, we first glue all extensions together to get a height function $h\in L^\infty((0,T_{sup})\times\Omega)$ which automatically satisfies $\|h\|_{L^\infty((0,T_{sup})\times\Omega)}=a^*$.
Similarly, we get the temperatures $(\Theta,\vartheta)$.
Now, $(\Theta,\vartheta,h)$ is a solution on $(0,T')$ for all $0<T'<T_{sup}$.
More specifically, it satisfies all conditions in \cref{def:weak_sol_fixed} for all smaller time intervals.
Since none of the estimates for $(\Theta,\vartheta,h)$ blow-up for $T'\to T_{sup}$ (note that $\|h\|_{L^\infty((0,T_{sup})\times\Omega)}\leq a^*$ and $T_{sup}<\infty$), if follows that $(\Theta,\vartheta,h)\in\mathcal{X}(T_{sup})\times \mathcal{H}(T_{sup})$ is a solution over $(0,T_{sup})$.
As a consequence, $T_{sup}=\infty$.
\end{proof}

\section{Numerical Analysis}\label{sec:numerical_analysis}
From a numerical perspective, the present model is quite challenging and computationally expensive to solve due to the different scales, changing domains, and non-linear coupling.
In this section, we introduce and investigate a precomputing strategy together with a semi-implicit time-stepping scheme to resolve these problems.

First note, that independent of the underlying numerical algorithm, the effective conductivity $K$ has to be recomputed multiple times throughout the simulation process: after each discrete time step and at every macroscopic node since $K$ depends on the local value of the height function $h$.
This is quite expensive since differential equations (cell problems) have to be solved to compute $K$.
Seeing that in the underlying setup, the general cell topology does not change for different heights and only depends on $h$, it would be advantageous if one could precompute the conductivity $K$ for a discrete number of height values in an offline phase and utilize these values in the simulation by interpolation.
We start by analyzing this idea and the influence of the solution in more detail.
\subsection{Precomputing strategy}
Given $N \in \N$, we evenly split the interval $[-a^*, a^*]$ into $N$ sub-intervals $[h_i,h_{i+1}]$ with $h_i = -a^* + 2a^*\tfrac{i}{N}$, for $i=0,\dots,N$.
We then compute for each $h_i$ the corresponding conductivity $K_i = K(h_i)$ given by \cref{eq:system:conductivity}.
We also set $\Delta h = \nicefrac{1}{N}$ in the following.\footnote{
    Depending on the expected behavior of $K(h)$ with respect to $h$, it might be advantageous to utilize a nonuniform grid (see e.g., \cite{Nepal2024}).
    In our case, this is unnecessary as the general behavior is quite smooth and uniform (for comparison, see \cref{fig:K_dependence_of_h}).
}
For any $h \in [-a^*, a^*]$ the precomputed conductivity $K_\text{pre}$ is then defined via an interpolation scheme. This raises the question of what interpolation scheme is suitable given the properties of $K$.
In line with the Lipschitz continuity of $K$ (\cref{lemma:lipschitz_parameters}), a piecewise linear interpolation is a natural candidate, e.g.,
\begin{equation}
    K_\text{lin}(h) = \left(1 - \frac{h - h_i}{\Delta h}\right)K_i +  \frac{h - h_i}{\Delta h} K_{i+1} , \quad \text{if } h \in [h_i, h_{i+1}),
\end{equation}
This Lipschitz continuity immediately guarantees at least a first-order interpolation error:
\begin{equation*}
     |K(h) - K_\text{lin}(h)| \leq C \frac{\Delta h}{2}.
\end{equation*}
%
We need higher regularity results for the effect conductivity to ensure higher order errors, which we investigate in the following lemma.
\begin{lemma}\label{lemma:K_differentiable}
    For the effective conductivity it holds $K \in C^{1,1}([-a^*, a^*], \R^{d\times d})$.
\end{lemma}
\begin{proof}
    Recall, that the effective conductivity is given by
    \begin{align*}
        K_{ij}(h) = \int_{Y\setminus Y(h)} \mathcal{K}(\nabla\xi_{h,j} + e_j) \cdot e_i\di{y}
        = \int_{Y\setminus Y_0} \mathcal{K} (F_h^{-1} \nabla(S_h^{-1}\xi_{h,j}) + e_j) \cdot e_i J_h\di{y}.
    \end{align*}
    Therefore, by the product rule, we obtain the following expression assuming that all parts are well-defined:
    \begin{equation}\label{eq:partial_of_K}
        \begin{split}
            \partial_h K_{ij}(h) 
            &= 
            \int_{Y\setminus Y_0} K (\partial_h F_h^{-1} \nabla(S_h^{-1}\xi_{h,j})) \cdot e_i J_h \di{y}
            + 
            \int_{Y\setminus Y_0} K (F_h^{-1} \partial_h \nabla(S_h^{-1}\xi_{h,j})) \cdot e_i J_h\di{y} \\
            &\hspace{50pt}+ 
            \int_{Y\setminus Y_0} K (F_h^{-1} \nabla(S_h^{-1}\xi_{h,j}) + e_j) \cdot e_i \partial_h J_h\di{y}.
        \end{split}
    \end{equation}
    Now, we investigate if all the derivatives on the right-hand side are well-defined and Lipschitz with respect to $h$.
    For the Matrix $F_h^{-1} = (\ID+h D_y\psi)^{-1}$, it holds $\partial_h F_h^{-1} = -F_h^{-1} D_y\psi F_h^{-1}$. 
    Therefore, $\partial_h F_h^{-1}$ is well-defined and also Lipschitz in $h$.
    For the determinate $J_h = \det(\ID+hD_y \psi)$, it holds
    \begin{align*}
        \partial_h J_h &= \partial_h\sum_{\sigma \in S_d} \text{sgn}(\sigma) \prod_{i=1}^d (\delta_{i,\sigma(i)} + h D_y\psi_{i, \sigma(i)}) 
        = \sum_{k=1}^d  \sum_{\sigma \in S_d} \text{sgn}(\sigma) \prod_{i=1}^d p_{k, i, \sigma(i)},
    \end{align*}
    where $S_d$ denotes the set of all permutations in dimension $d$, $\text{sgn}(\sigma)=\pm1$ depending on if the permutation is even or odd, and
    \begin{equation*}
        p_{k, i, \sigma(i)} = 
        \begin{cases}
            (\delta_{i,\sigma(i)} + h D_y\psi_{i, \sigma(i)}), &\text{ if } i\neq k, \\
            D_y\psi_{i, \sigma(i)}, &\text{ if } i = k.
        \end{cases}
    \end{equation*}
    This is a sum of multiple determinants (each time removing the identity and $h$ from a different column of the matrix $\ID+hD_y \psi$), hence $\partial_h J_h$ is also well-defined and by the same arguments as in \cref{lemma:lipschitz_cell_functions} also Lipschitz with respect to $h$.

    Next, we investigate the derivative of the cell solution $S_h^{-1}\xi_{h, j}$ with respect to $h$. To this end, we define the operator $E_j:[-a^*, a^*] \times V \to V^*$, where 
    \begin{equation*}
        V \coloneqq \left\{v \in H^1_\#(Y\setminus Y_0) : \int_{Y \setminus Y_0} v \di{y}  = 0 \right\},
    \end{equation*}
    and the operator is given by
    \begin{equation*}
        E_j(h, v) = \int_{Y \setminus Y_0} J_h F^{-1}_h \mathcal{K} F^{-T}_h \nabla_y v \cdot \nabla_y \ast \di{y} - \int_{\Gamma_0} J_h e_j \cdot n \ast \di{y}.
    \end{equation*}
    Here $\ast$ denotes a placeholder, for the argument of the functional.
    Note, that the cell solutions $S_h^{-1}\xi_{h, j}$ are implicitly defined by $E_j(h, S_h^{-1}\xi_{h, j}) = 0$.
    We want to utilize the implicit function theorem \cite[Theorem 1.41]{Hinze2009} to show the existence of the Fr\'echet-derivative of the cell solution and find a characterization of it.
    Given that $E_j$ is linear in $v\in V$ and $F_h, J_h$ are differentiable in $h\in[-a^*,a^*]$, it is straightforward to check that $E_j$ is Fr\'echet-differentiable. The derivative $\partial_u E_j(h, S_h^{-1}\xi_{h, j})$ is a bounded linear operator from $V \to V^*$ and defined by 
    \begin{equation*}
        \partial_u E_j(h, S_h^{-1}\xi_{h, j})[v] = \int_{Y \setminus Y_0} J_h F^{-1}_h \mathcal{K} F^{-T}_h \nabla_y v \cdot \nabla_y \ast \di{y}.
    \end{equation*}
    For $v = 0 \in V$ it holds $\|\partial_u E_j(h, S_h^{-1}\xi_{h, j})[v]\|_{V^*} = 0$ and for $v \neq 0$
    \begin{align*}
        \|\partial_u E_j(h, S_h^{-1}\xi_{h, j})[v]\|_{V^*} 
        = \hspace{-3pt} \sup_{\varphi \in V, \varphi \neq 0} \hspace{-3pt}\frac{|\partial_u E_j(h, S_h^{-1}\xi_{h, j})[v] (\varphi) |}{\|\varphi\|_{H^1(Y\setminus Y_0)}} 
        \geq \frac{|\partial_u E_j(h, S_h^{-1}\xi_{h, j})[v] (v) |}{\|v \|_{H^1(Y\setminus Y_0)}} 
        \geq C \|v \|_{H^1(Y\setminus Y_0)},
    \end{align*}
    where in the last inequality we applied assumption (A1), \cref{lemma:lipschitz_transformation} (i) and the Poincar\'e inequality. Therefore, $\|\partial_u E_j(h, S_h^{-1}\xi_{h, j})[v]\|_{V^*} \geq C \|v \|_{H^1(Y\setminus Y_0)}$ for all $v \in V$. With \cite[Theorem 6.7.1]{suhubi2003} we can conclude that $\partial_u E_j(h, S_h^{-1}\xi_{h, j})$ is injective and the image is closed, hence the inverse exists. Additionally, the inverse is also bounded. By the implicit function theorem, we obtain that the solution operator $h \mapsto S_h^{-1}\xi_{h, j}$ is well-defined and Fr\'echet-differentiable, where the derivative is given by
    \begin{equation*}
        \begin{aligned}
            &               &  \partial_u E_j(h, S_h^{-1}\xi_{h, j}) [\partial_h (S_h^{-1}\xi_{h, j})] &= -\partial_h E_j(h, S_h^{-1}\xi_{h, j}) 
            \\
            &\Leftrightarrow& \int_{Y\setminus Y_0} J_h F^{-1}_h \mathcal{K} F^{-T}_h \nabla_y (\partial_h (S_h^{-1}\xi_{h, j})) \cdot \nabla_y\ast \di{y} &= - \int_{Y\setminus Y_0} \partial_h\left(J_h F^{-1}_h \mathcal{K} F^{-T}_h\right) \nabla_y (S_h^{-1}\xi_{h, j}) \cdot \nabla_y\ast \di{y} 
            \\
            & & &+ \int_{\Gamma_0}\partial_h J_h e_j\cdot n \ast \di{y},
        \end{aligned}
    \end{equation*}
    understood as an element in $V^*$. With the same arguments as for the estimate in \cref{eq:lipschitz_xi_h} one can show that $\partial_h (S_h^{-1}\xi_{h, j})$ is Lipschitz in $h$ as well. 
    Finally, we conclude that $\partial_h K_{ij}$ is well-defined and Lipschitz in $h$ since all parts of the right-hand side of \cref{eq:partial_of_K} are bounded and Lipschitz. 
\end{proof}
The above regularity ensures a quadratic error order if we utilize linear interpolation:
\begin{equation*}
    |K(h) - K_\text{lin}(h)| \leq C (\Delta h)^2\quad
    \text{for all } h \in [-a^*, a^*].
\end{equation*}
\begin{remark}\label{remark:Lipschitz_constant_of_derivative}
    Since $F_h^{-1}$ and $J_h$ are smooth regarding $h$, the implicit function theorem can also be used to show the existence of higher derivatives of the solution operator. Therefore, higher-order interpolation schemes can be applied as well. But we have to note, that the Lipschitz constant grows exponentially with $\nicefrac{1}{a^*}$. E.g., for the first derivative we see in the above proof, that $\| \partial_h F_h^{-1}\| \leq C \|D_y \psi\| \leq \nicefrac{5C}{a^*}$, which leads to
    \begin{equation*}
        |\partial_h K(h + \Delta h) - \partial_h K(h)| \leq C \left(1 + \frac{1}{(a^*)^2} \right) \Delta h.
    \end{equation*}
    This means that high-order interpolations may need a small step size $\Delta h$ to ensure the desired convergence order.
\end{remark}
\begin{remark}
    By assumption (A1) we have that the underlying conduction matrix $\mathcal{K}$ is not dependent on the space variable. For some applications this may not hold and $\mathcal{K}$ may be a function of $y$. In this case the results of \cref{sec:analysis} are still valid under the condition that $\mathcal{K}$ is Lipschitz in regard to $y$. The property $K \in C^{1,1}([-a^*, a^*], \R^{d\times d})$ requires the additional assumption that $\mathcal{K} \in C^{1,1}(Y\setminus Y_0)$. 
\end{remark}
In the following, $K_\text{pre}$ stands for the precomputed conductivity, while $\mathcal{E}_\text{pre}:= \sup_{|h|\leq a^*}|K(h)-K_\text{pre}(h)|$ denotes the corresponding interpolation error.
In the case that $K_\text{pre} = K_\text{lin}$, or a higher order interpolation, $K_\text{pre}$ is a Lipschitz continuous function in $h$ and the results of \cref{sec:analysis} can be transferred to ensure the existence of solutions for the precomputed system.
The important question is how the interpolation error propagates over time, this is analyzed in the following theorem.
\begin{theorem}[Error propagation of precomputation]\label{thm:error_precomputing}
    Let $(h, \Theta, \vartheta)$ denote the solution with the exact conductivity and $(h_\text{pre}, \Theta_\text{pre}, \vartheta_\text{pre})$ the solution with the precomputed values $K_\text{pre}$. Set $\Bar{T}^* = \min(T^*, T_\text{pre}^*)$ such that both solutions are defined on $(0, \Bar{T}^*)$. Assume that 
    $\partial_t \Theta_\text{pre}, \nabla \Theta_\text{pre} \in L^\infty((0, \Bar{T}^*) \times \Omega)$, 
    $\partial_t \vartheta_\text{pre}, \nabla_y \vartheta_\text{pre} \in L^\infty((0, \Bar{T}^*) \times \Omega \times Y_0)$. Then, there exists a $C > 0$ (depending on $\Bar{T}^*$) such that
    \begin{equation*}
        \|\Bar{\Theta}\|_{L^\infty((0, \Bar{T}^*); L^2(\Omega))} + \|\Bar{\vartheta}\|_{L^\infty((0, \Bar{T}^*); L^2(\Omega \times Y_0))} 
        + \|\nabla \Bar{\Theta}\|_{L^2((0, \Bar{T}^*) \times \Omega)}
        + \|\nabla \Bar{\vartheta}\|_{L^2((0, \Bar{T}^*) \times \Omega \times Y_0)}
        \leq C \mathcal{E}_\text{pre}, 
    \end{equation*}
    where $(\Bar{h}, \Bar{\Theta}, \Bar{\vartheta}) \coloneqq (h-h_\text{pre}, \Theta - \Theta_\text{pre}, \vartheta - \vartheta_\text{pre})_{|(0, \Bar{T}^*)}$ denotes the difference of both solutions.
\end{theorem}
\begin{proof}
    The proof follows the same idea as the uniqueness proof in \cref{lemma:uniqueness}. Here, we only have to consider the macroscopic conductivity term in more detail, all other terms can be estimated as previously shown. For the integral $I_3$ in the proof of \cref{lemma:uniqueness}, we have
    \begin{align*}
        I_3 = 
        \int_{\Omega}K(h)\nabla\bar{\Theta}\cdot\nabla\bar{\Theta}\di{x}
        + 
        \int_{\Omega} \left[K(h) - K_\text{pre}(h_\text{pre})\right] \nabla \Theta_\text{pre} \nabla\bar{\Theta}\di{x}.
    \end{align*}
    The integral on the right can be further estimated by
    \begin{align*}
        \int_{\Omega} &| \left[K(h) - K_\text{pre}(h_\text{pre})\right] \nabla\Theta_\text{pre} \cdot\nabla\bar{\Theta}|\di{x} 
        \\
        &\leq \int_{\Omega}|(K(h)  - K(h_\text{pre} )) \nabla\Theta_\text{pre}\cdot\nabla\bar{\Theta}|\di{x}
        + \int_{\Omega}|(K(h_\text{pre} ) - K_\text{pre}(h_\text{pre})) \nabla\Theta_\text{pre} \cdot\nabla\bar{\Theta}|\di{x} 
        \\
        &\leq C \|\nabla \Theta_\text{pre}\|_{L^\infty(\Omega)} \|\nabla \Bar{\Theta}\|_{L^2(\Omega)}\left(\|\Bar{h}\|_{L^2(\Omega)} + \mathcal{E}_\text{pre} \right).
    \end{align*}
    Now plugging this estimate into the proof steps of \cref{lemma:uniqueness}, we obtain, for any time interval $(0, t) \subset (0, \bar{T}^*)$, the following estimate
    \begin{multline*}
        \|\bar{\Theta}(t, \cdot)\|_{L^2(\Omega)}^2
        +\|\bar{\vartheta}(t, \cdot)\|_{L^2(\Omega\times Y_0)}^2
        +\|\nabla\bar{\Theta}\|^2_{L^2((0,t) \times \Omega)}
        +\|\nabla_y\bar{\vartheta}\|^2_{L^2((0,t) \times \Omega \times Y_0)}
        \\
        \leq C \mathcal{E}_\text{pre}^2 + C (1 + t^2) \left(\|\bar{\Theta}\|_{L^2((0, t) \times \Omega)}^2 + \|\bar{\vartheta}\|_{L^2((0, t) \times \Omega \times Y_0)}^2\right).
    \end{multline*}
    Applying the Grönwall inequality leads to the desired result.    
\end{proof}
\subsection{Resolving the nonlinearity.} Even with the precomputing strategy at hand, we are still facing a nonlinear two-scale problem. We propose a semi-implicit time-stepping method to resolve the nonlinearity, which is often applied to nonlinear equations in numerical simulations. To this end, given a time step $\Delta t$ we define the discrete time point $t_i = i \Delta t$. Then we are looking at each $i > 0$ for the time-discrete solution $(\Theta_i, \vartheta_i) \in \mathrm{X}$ of
\begin{equation}\label{eq:time_discrete_equation}
    \begin{split}
        &\int_\Omega C(h_i) \frac{\Theta_i - \Theta_{i-1}}{\Delta t} \Phi \di{x}
        + \int_\Omega \int_{Y_0} c(h_i)  \frac{\vartheta_i - \vartheta_{i-1}}{\Delta t} \phi \di{y} \di{x}
        + \int_\Omega K_\text{pre}(h_i) \nabla \left(\frac{\Theta_i + \Theta_{i-1}}{2}\right) \cdot \nabla \Phi \di{x} 
        \\
        &+ \int_\Omega \int_{Y_0} \kappa(h_i) \nabla_y \frac{\vartheta_i + \vartheta_{i-1}}{2} \cdot \nabla_y \phi \di{y} \di{x}
        +  \hspace{-2pt}\int_\Omega\int_{Y_0} \vartheta_i v(h_i)\cdot\nabla_y\phi\di{y}\di{x}
        = \hspace{-2pt}\int_{\Omega} \left(F_i- L(h_i)\frac{h_i - h_{i-1}}{\Delta t}\right)\Phi\di{x}
        \\
        &+\int_\Omega\int_{Y_0}f(h_i)\phi\di{y}\di{x} 
        - \int_{\Omega} \partial_h C(h_i) \frac{h_i - h_{i-1}}{\Delta t} \Theta_{i-1} \Phi \di{x}
        - \int_{\Omega} \partial_h c(h_i) \frac{h_i - h_{i-1}}{\Delta t} \vartheta_{i-1} \phi \di{x}
    \end{split}
\end{equation}
for all $(\Phi,\phi) \in \mathrm{X}$ where 
\begin{equation*}
    \mathrm{X} \coloneqq \left\{
        (\Phi,\phi) \in H^1(\Omega) \times L^2(\Omega, H^1(Y_0)) : \Phi(x) = \phi(x, y) \text{ a.e. on } \Omega \times \Gamma_0 
    \right\}.
\end{equation*}
To find $(\Theta_i, \vartheta_i)$ a differential equation has to be solved at each time step. But we now linearize the problem by using an explicit scheme to compute the height function $h$, to be precise
\begin{equation*}
    h_i = h_{i-1} + \Delta t (\Theta_{i-1} - \Theta^{ref}) \quad \text{ for } i > 0 \quad \text{and} \quad h_0 = 0.
\end{equation*}
This also leads to the simplification of $\nicefrac{(h_i - h_{i-1})}{\Delta t} = \Theta_{i-1} - \Theta^{ref}$ in the above scheme. In (\ref{eq:time_discrete_equation}) we used the product rule to split the time derivative into the derivative of $\Theta$ and of $C$, respectively for $\vartheta$ and $c$. This simplifies the proof for the time stepping error. Note, that also the more natural discretization 
\begin{equation*}
    \partial_t(C(h) \Theta) \approx \frac{C(h_i) \Theta_i - C(h_{i-1}) \Theta_{i-1}}{\Delta t}
\end{equation*}
could be used and only deviates from the above scheme by a term of order $\Delta t$. For the gradients, we utilize an average between the current and previous time step since this fits more naturally in the following proof idea. Note, that this scheme requires $\nabla \Theta_0 \in L^2(\Omega)$ and $\nabla_y \vartheta_0 \in L^2(\Omega \times Y_0)$. But again, just $\nabla \Theta_i$, and $\nabla_y \vartheta_i$, could be used.

If $|h_i| \leq a^*$, the solution of the system (\ref{eq:time_discrete_equation}) exists and is unique since it is a linear time-discretized parabolic two-scale equation. Given that $\Theta_0 \in L^\infty(\Omega)$ we can always carry out one time step if $\Delta t$ is small enough. To show that the discrete solution exists on a larger time interval $(0, T^*_\text{dis})$, we need to show that the discrete solution $\Theta_i$ can be bounded in $L^\infty(\Omega)$ independent of the height function. For this, we can proceed as in \cref{lemma:wellposedness_given_h}, e.g., we would need to carry out the following steps:
\begin{enumerate}
    \item Choose a $M \geq 1$ and a $N \in \N$, then define the time horizon $T^*_\text{dis}=\nicefrac{a^*}{M^4}$ and step $\Delta t = \nicefrac{T^*_\text{dis}}{N}$.
    \item Start with a fixed height function $\{h_i\}_{i=0}^N$ defined by 
    \begin{equation*}
        h_i = h_{i-1} + \Delta t b_i, \quad h_0 = 0 \quad \text{ and }
        \quad \|b_i\|_{L^\infty} \leq M.
    \end{equation*}
    \item Derive a $L^\infty(\Omega)$ bound for the solution of \cref{eq:time_discrete_equation} for the prescribed height function $\{h_i\}_{i=0}^N$ independent of $M$. This follows with similar computations as in \cref{lemma:wellposedness_given_h} and by testing with $\Phi = \Theta_i + \Theta_{i-1}, \phi = \vartheta_i + \vartheta_{i-1}$ and summing over all time steps. As the proof is similar to the previous one, it will not be presented in detail here.
    \item Choose the $M$ corresponding to the solution bound and obtain that the discrete method is reasonable and the solution exists on a nonempty interval $(0, T^*_\text{dis})$.
\end{enumerate}
Note, that the above procedure also ensures that for a smaller time step $\Delta t$ the discrete solution exists, at least, on the same time interval. The error we introduce through this scheme is analyzed in the next lemma.
\begin{lemma}[Error of time stepping scheme]\label{lemma:time_stepping}
    Let $(\hat{h}, \hat{\Theta}, \hat{\vartheta})$ denote the solution that is continuous in time, exists on the time interval $(0, T^*)$ and uses the precomputed values $K_\text{lin}$. Denote with $\{h_i, \Theta_i, \vartheta_i\}_{i=0}^N$ the time discrete solution of \cref{eq:time_discrete_equation} that exists on $[0, T^*_\text{dis}]$, with $K_\text{pre} = K_\text{lin}$. Let the assumptions

    \vspace{\medskipamount}
    \begin{minipage}[t]{0.49\linewidth}
      \begin{itemize}
        \item $\hat{\Theta} \in W^{2, \infty}((0, T^*); W^{1, \infty}(\Omega))$,
        \item $F \in C(\overline{S \times \Omega})$,
        \item $\Delta t \leq \nicefrac{C(a^*)}{8K(a^*)}$,
      \end{itemize}
    \end{minipage}
    \begin{minipage}[t]{0.49\linewidth}
      \begin{itemize}
        \item $\hat{\vartheta} \in W^{2, \infty}((0, T^*); W^{1, \infty}(\Omega \times Y_0))$,
        \item $f \in C(\overline{S \times \Omega}; C^{0, 1}(\overline{Y_0}))$,
      \end{itemize}
    \end{minipage}
    \vspace{\medskipamount}
    
    \noindent be fulfilled. Let $N_t \in \N$ be such that $\Delta t N_t < \min{(T^*, T^*_\text{dis})}$ and $\Delta t(1 + N_t) \geq \min{(T^*, T^*_\text{dis})}$. Then, there exists a $C > 0$ (depending on $(\hat{\Theta}, \hat{\vartheta})$) such that
    \begin{equation*}
        \sup_{i=1,\dots, N_t} \|\hat{\Theta}(t_i) - \Theta_i\|_{H^1(\Omega)} + \|\hat{\vartheta}(t_i) - \vartheta_i\|_{H^1(\Omega)} + \|\hat{h}(t_i) - h_i\|_{L^2(\Omega)} \leq C (\Delta t + \Delta h). 
    \end{equation*}
\end{lemma}
\begin{proof} 
    We follow the proof ideas of \cite[Section 3 and 4]{Ewing1978} where numerical approaches for similar nonlinear problems are analyzed. The main difference in our case is, that we consider a two-scale system, the nonlinearity is coupled to the height function $h$, a less regular conductivity, and a different right-hand side because of $C(h)$. But the general idea from \cite{Ewing1978} can be transferred to our case. We denote by $\Bar{\cdot}_i$ the difference between the continuous and discrete solution at time point $t_i$, e.g., $\Bar{\Theta}_i = \hat{\Theta}(t_i) - \Theta_i$. We start by investigating the difference between the height functions of both solutions. By definition of both functions, we obtain
    \begin{align*}
        \|\Bar{h}_i\|_{L^2(\Omega)} 
        &= \|\hat{h}(t_i) - h_i\|_{L^2(\Omega)}
        \\
        &\leq \|\hat{h}(t_{i-1}) - h_{i-1}\|_{L^2(\Omega)} 
        +  \left \| \int_{t_{i-1}}^{t_i} (\hat{\Theta}(s) - \Theta^{ref}) \di{s} - \Delta t (\Theta_{i-1} - \Theta^{ref}) \right \|_{L^2(\Omega)} 
        \\
        &\leq \|\hat{h}(t_{i-1}) - h_{i-1}\|_{L^2(\Omega)} 
        + \Delta t \|\Bar{\Theta}_{i-1}\|_{L^2(\Omega)}
        + \left \| \int_{t_{i-1}}^{t_i} \int_{t_{i-1}}^s \partial_t \hat{\Theta}(\tau) \di{\tau} \di{s} \right \|_{L^2(\Omega)}
        \\
        &\leq \|\hat{h}(t_{i-1}) - h_{i-1}\|_{L^2(\Omega)} 
        + \Delta t \|\Bar{\Theta}_{i-1}\|_{L^2(\Omega)} + C (\Delta t )^2
    \end{align*}
    in the estimates we used the fundamental theorem of calculus and that $\partial_t \hat{\Theta}$ is bounded. Iteratively applying the above argument leads to
    \begin{equation}\label{eq:proof_h_estimate}
        \|\Bar{h}_i\|_{L^2(\Omega)} \leq \sum_{n = 1}^{i-1} \left[ \Delta t \|\Bar{\Theta}_{n}\|_{L^2(\Omega)} + C (\Delta t )^2 \right].
    \end{equation}
    Next, we derive the estimates for the temperature. To this end, we only show the error estimate for the macroscopic function $\Theta$. The reason for this is, that the estimates for $\vartheta$ follow with the same steps since the microscopic equation has a similar structure. This allows us to reduce the length of the proof considerably. Therefore, let the continuous solution $\Hat{\Theta}$ fulfill
    \begin{equation}\label{eq:proof_time_stepping_cont_eq}
        \int_\Omega C(\Hat{h}) \partial_t \Hat{\Theta} \Phi \di{x} 
        + \int_\Omega \partial_h C(\Hat{h}) \partial_t \Hat{h} \Hat{\Theta} \Phi \di{x}
        + \int_\Omega K_\text{lin}(\Hat{h}) \nabla \Hat{\Theta} \cdot \nabla \Phi \di{x} 
        = \int_\Omega  L(\Hat{h}) \partial_t \Hat{h} \Phi \di{x}
    \end{equation}
    and the time-discrete solution 
    \begin{equation}\label{eq:proof_time_stepping_discrete_eq}
        \begin{split}
            \int_\Omega C(h_i) \frac{\Theta_i - \Theta_{i-1}}{\Delta t} &\Phi \di{x}
            + \int_{\Omega} \partial_h C(h_i) (\Theta_{i-1} - \Theta^{ref})  \Theta_{i-1} \Phi \di{x}
            \\
            &+ \int_\Omega K_\text{lin}(h_i) \nabla \left(\frac{\Theta_i + \Theta_{i-1}}{2}\right) \cdot \nabla \Phi \di{x}  
            = \int_{\Omega}  L(h_i) (\Theta_{i-1} - \Theta^{ref}) \Phi \di{x}.
        \end{split}
    \end{equation}
    To derive an error estimate at  $t_i$, we consider (\ref{eq:proof_time_stepping_cont_eq}) localized at $t_i$, subtract  the \cref{eq:proof_time_stepping_cont_eq} and (\ref{eq:proof_time_stepping_discrete_eq}) from each other and test with $\Phi = \Bar{\Theta}_i - \Bar{\Theta}_{i-1}$. We obtain four terms which we estimate individually. Let us introduce the discrete time derivative $d_t$ which is defined by $d_t \Theta_i = \nicefrac{\Theta_i - \Theta_{i-1}}{\Delta t}$. Additionally, we abuse the notation and write $d_t \hat{\Theta}_i = \nicefrac{\hat{\Theta}(t_i) - \hat{\Theta}(t_{i-1})}{\Delta t}$. We start with the time derivative term, which yields
    \begin{align*}
        \int_\Omega &C(h_i) d_t\Theta_i \Phi \di{x} - \int_\Omega C(\Hat{h}) \partial_t \Hat{\Theta}(t_i) \Phi \di{x} 
        = \underbrace{\int_\Omega C(h_i) d_t\Theta_i \Phi \di{x} 
        - \int_\Omega C(h_i) d_t\hat{\Theta}_i \Phi \di{x}}_{=I_1}
        \\
        &+ \underbrace{\int_\Omega C(h_i) d_t\hat{\Theta}_i \Phi \di{x} 
        - \int_\Omega C(\Hat{h}(t_i)) d_t\hat{\Theta}_i \Phi \di{x}}_{=I_2}
        + \underbrace{\int_\Omega C(\Hat{h}(t_i)) d_t\hat{\Theta}_i \Phi \di{x}
        - \int_\Omega C(\Hat{h}(t_i)) \partial_t \Hat{\Theta}(t_i) \Phi \di{x}}_{=I_3}.
    \end{align*}
    For terms on the right-hand side, it holds
    \begin{align*}
        I_1 &= \Delta t \int_{\Omega} C(h_i) (d_t \Bar{\Theta}_i)^2 \di{x} \geq \Delta t C(a^*)  \|d_t \Bar{\Theta}_i\|_{L^2(\Omega)}^2,
        \\
        |I_2| &\leq C \|\Bar{h}_i\|_{L^2(\Omega)} \|\Bar{\Theta}_i - \Bar{\Theta}_{i-1}\|_{L^2(\Omega)},
        \\
        |I_3| &\leq C \Delta t \|\partial_{tt} \hat{\Theta}(s) \di{s} \|_{L^\infty((0, T^*) \times \Omega)} \|\Bar{\Theta}_i - \Bar{\Theta}_{i-1}\|_{L^2(\Omega)} \leq C (\Delta t)^3 + \e \Delta t \|d_t \Bar{\Theta}_i\|_{L^2(\Omega)}^2.
    \end{align*}
    Here, we used for $I_2$ and $I_3$ that $C$ is Lipschitz, the Hölder inequality, and the Young inequality with a small $\e > 0$. $I_2$ can be estimated further by utilizing \cref{eq:proof_h_estimate},
    \begin{align*}
        |I_2|
        &\leq  C \sum_{n = 1}^{i-1} \left[ \Delta t \|\Bar{\Theta}_{n}\|_{L^2(\Omega)} + C (\Delta t )^2 \right] \|\Bar{\Theta}_i - \Bar{\Theta}_{i-1}\|_{L^2(\Omega)} 
        \\
        &= C \sum_{n = 1}^{i-1} \left[ \Delta t \|\Bar{\Theta}_{n}\|_{L^2(\Omega)} \Delta t \|d_t \Bar{\Theta}_i\|_{L^2(\Omega)} + C (\Delta t )^2 \Delta t \|d_t \Bar{\Theta}_i\|_{L^2(\Omega)} \right] 
        \\
        &\leq C \sum_{n = 1}^{i-1} \left[(\Delta t)^2 \|\Bar{\Theta}_{n}\|_{L^2(\Omega)}^2 + \e (\Delta t)^2 \|d_t \Bar{\Theta}_i\|_{L^2(\Omega)}^2 + C (\Delta t )^4 + \e (\Delta t)^2 \|d_t \Bar{\Theta}_i\|_{L^2(\Omega)}^2 \right] 
        \\ &\leq  C (\Delta t )^3 + C \e \Delta t \|d_t \Bar{\Theta}_i\|_{L^2(\Omega)}^2 + C \Delta t \sum_{n = 1}^{i-1} \Delta t \|\Bar{\Theta}_{n}\|_{L^2(\Omega)}^2.  
    \end{align*}
    Next, the term with $\partial_h C$, where we obtain
    \begingroup
    \allowdisplaybreaks
    \begin{align*}
        \int_{\Omega} \partial_h C(h_i) &(\Theta_{i-1} - \Theta^{ref})  \Theta_{i-1} \Phi \di{x} - \int_\Omega \partial_h C(\Hat{h}(t_i)) (\Hat{\Theta}(t_i) - \Theta^{ref}) \Hat{\Theta}(t_i) \Phi \di{x} 
        \\
        &= \underbrace{\int_{\Omega} \partial_h C(h_i) (\Theta_{i-1} - \Theta^{ref})  \Theta_{i-1} \Phi \di{x} 
        - \int_\Omega \partial_h C(h_i) (\Theta_{i-1} - \Theta^{ref})  \hat{\Theta}(t_{i-1}) \Phi \di{x}}_{=I_4}
        \\
        &+ \underbrace{\int_\Omega \partial_h C(h_i) (\Theta_{i-1} - \Theta^{ref})  \hat{\Theta}(t_{i-1}) \Phi \di{x} -
        \int_\Omega \partial_h C(h_i) (\hat{\Theta}(t_{i-1}) - \Theta^{ref})  \hat{\Theta}(t_{i-1}) \Phi \di{x}}_{=I_5}
        \\
        &+ \underbrace{
        \int_\Omega \partial_h C(h_i) (\hat{\Theta}(t_{i-1}) - \Theta^{ref})  \hat{\Theta}(t_{i-1}) \Phi \di{x} - \int_\Omega \partial_h C(\Hat{h}(t_i)) (\Hat{\Theta}(t_{i-1}) - \Theta^{ref}) \hat{\Theta}(t_{i-1})\Phi \di{x} }_{=I_6}
        \\
        &+ \underbrace{
        \int_\Omega \partial_h C(\Hat{h}(t_i)) (\hat{\Theta}(t_{i-1}) - \Theta^{ref})  \hat{\Theta}(t_{i-1}) \Phi \di{x} - \int_\Omega \partial_h C(\Hat{h}(t_i)) (\Hat{\Theta}(t_i) - \Theta^{ref}) \Hat{\Theta}(t_i) \Phi \di{x} }_{=I_7}
    \end{align*}
    \endgroup
    Using that $\partial_h C$ is well-defined and bounded (since $\partial_h J_h$ is well-defined and bounded; see the proof of \cref{lemma:K_differentiable}) and that the solutions are bounded, we can obtain with similar arguments as before
    \begin{align*}
        |I_4| + |I_5| 
        &\leq C \|\Bar{\Theta}_{i-1}\|_{L^2(\Omega)}  \|\Bar{\Theta}_i - \Bar{\Theta}_{i-1}\|_{L^2(\Omega)}
        \leq C \Delta t \|\Bar{\Theta}_{i-1}\|_{L^2(\Omega)}^2 + \e \Delta t  \|d_t \Bar{\Theta}_i\|_{L^2(\Omega)}^2,
        \\
        |I_6| &\leq C \|\Bar{h}_i\|_{L^2(\Omega)} \|\Bar{\Theta}_i - \Bar{\Theta}_{i-1}\|_{L^2(\Omega)},
        \\
        |I_7| &\leq C \Delta t \|\partial_t \hat{\Theta}\|_{L^\infty((0, T^*) \times \Omega)} \|\Bar{\Theta}_i - \Bar{\Theta}_{i-1}\|_{L^2(\Omega)} \leq C (\Delta t)^3 + \e \Delta t \|d_t \Bar{\Theta}_i\|_{L^2(\Omega)}^2.
    \end{align*}
    Since the term containing $ L$ is comparable to the previous ones, we first handle it and afterward estimate the gradients. By introducing the mixed terms, we have the following equality 
    \begingroup
    \allowdisplaybreaks
    \begin{align*}
        \int_{\Omega}  L(h_i) (\Theta_{i-1} - \Theta^{ref}) \Phi \di{x} &- \int_\Omega  L(\Hat{h}(t_i)) (\hat{\Theta}(t_{i}) - \Theta^{ref}) \Phi \di{x} 
        \\
        &=  \underbrace{\int_{\Omega}  L(h_i) (\Theta_{i-1} - \Theta^{ref}) \Phi \di{x} - \int_{\Omega}  L(h_i) (\Hat{\Theta}(t_{i-1}) - \Theta^{ref}) \Phi \di{x}}_{=I_8}
        \\ 
        &+ \underbrace{\int_{\Omega}  L(h_i) (\Hat{\Theta}(t_{i-1}) - \Theta^{ref}) \Phi \di{x} - \int_{\Omega}  L(h_i) (\Hat{\Theta}(t_{i}) - \Theta^{ref}) \Phi \di{x}}_{=I_9}
        \\
        &+
        \underbrace{\int_{\Omega}  L(h_i) (\Hat{\Theta}(t_{i}) - \Theta^{ref}) \Phi \di{x} - \int_\Omega  L(\Hat{h}(t_i)) (\hat{\Theta}(t_{i}) - \Theta^{ref}) \Phi \di{x}}_{=I_{10}}
    \end{align*}
    \endgroup
    All of $I_8, I_9$ and $I_{10}$ are of the same structure as before and we can find the following estimates
    \begin{align*}
        |I_8| &\leq  C \Delta t \|\Bar{\Theta}_{i-1}\|_{L^2(\Omega)}^2 + \e \Delta t  \|d_t \Bar{\Theta}_i\|_{L^2(\Omega)}^2,
        \\
        |I_9| &\leq C (\Delta t)^3 + \e \Delta t \|d_t \Bar{\Theta}_i\|_{L^2(\Omega)}^2,
        \\
        |I_{10}| &\leq C \|\Bar{h}_i\|_{L^2(\Omega)} \|\Bar{\Theta}_i - \Bar{\Theta}_{i-1}\|_{L^2(\Omega)}.
    \end{align*}
    Let us next consider the gradient terms, here the estimates are more involved but follow the general idea outlined in \cite{Ewing1978}. Introducing mixed terms into the difference first leads to
    \begingroup
    \allowdisplaybreaks
    \begin{align*}
        \int_\Omega &K_\text{lin}(h_i) \nabla \left(\frac{\Theta_i + \Theta_{i-1}}{2}\right) \cdot \nabla \Phi \di{x} 
        - \int_\Omega K_\text{lin}(\Hat{h}(t_i)) \nabla \Hat{\Theta}(t_i) \cdot \nabla \Phi \di{x}
        \\
        &= \underbrace{\int_\Omega K_\text{lin}(h_i) \nabla \left(\frac{\Theta_i + \Theta_{i-1}}{2}\right) \cdot \nabla \Phi \di{x} 
        - \int_\Omega K_\text{lin}(h_i) \nabla \left(\frac{\hat{\Theta}(t_i) + \hat{\Theta}(t_{i-1})}{2}\right) \cdot \nabla \Phi \di{x}}_{=I_{11, i}}
        \\
        &+ \underbrace{ \int_\Omega K_\text{lin}(h_i) \nabla \left(\frac{\hat{\Theta}(t_i) + \hat{\Theta}(t_{i-1})}{2}\right) \cdot \nabla \Phi \di{x} 
        -  \int_\Omega K_\text{lin}(\hat{h}(t_i)) \nabla \left(\frac{\hat{\Theta}(t_i) + \hat{\Theta}(t_{i-1})}{2}\right) \cdot \nabla \Phi \di{x}}_{= I_{12, i}}
        \\
        &+ \underbrace{\int_\Omega K_\text{lin}(\hat{h}(t_i)) \nabla \left(\frac{\hat{\Theta}(t_i) + \hat{\Theta}(t_{i-1})}{2}\right) \cdot \nabla \Phi \di{x} - \int_\Omega K_\text{lin}(\Hat{h}(t_i)) \nabla \Hat{\Theta}(t_i) \cdot \nabla \Phi \di{x}}_{=I_{13, i}}. 
    \end{align*}
    \endgroup
    Starting with $I_{11, i}$, where we obtain
    \begin{align*}
        2 I_{11, i} &= \int_\Omega K_\text{lin}(h_i) \nabla (\Bar{\Theta}_i + \Bar{\Theta}_{i-1}) \cdot \nabla (\Bar{\Theta}_i - \Bar{\Theta}_{i-1}) \di{x} =  \int_\Omega K_\text{lin}(h_i) (\nabla \Bar{\Theta}_i \cdot \nabla \Bar{\Theta}_i - \nabla\Bar{\Theta}_{i-1} \cdot \nabla \Bar{\Theta}_{i-1}) \di{x}.
    \end{align*}
    To further estimate the error, we now sum over $m=1, \dots, i$, and using summation by parts we obtain
    \begin{align*}
        \sum_{m=1}^i \int_\Omega K_\text{lin}(h_i) (\nabla \Bar{\Theta}_i \cdot \nabla \Bar{\Theta}_i - \nabla\Bar{\Theta}_{i-1} \cdot \nabla \Bar{\Theta}_{i-1}) \di{x} &= \sum_{m=1}^{i-1} \int_\Omega \nabla \Bar{\Theta}_m \cdot \nabla \Bar{\Theta}_m (K_\text{lin}(h_m)  - K_\text{lin}(h_{m+1})) \di{x}
        \\ 
        &\hspace{50pt} + \int_\Omega \nabla \Bar{\Theta}_i \cdot \nabla \Bar{\Theta}_i K_\text{lin}(h_i) \di{x}.
    \end{align*}
    The sum on the right can be further estimated by 
    \begin{equation*}
        \left | \sum_{m=1}^{i-1} \int_\Omega \nabla \Bar{\Theta}_m \cdot \nabla \Bar{\Theta}_m (K_\text{lin}(h_m)  - K_\text{lin}(h_{m+1})) \di{x} \right | \leq  C \sum_{m=1}^{i-1} \Delta t  \|\nabla \Bar{\Theta}_m\|_{L^2(\Omega)}^2,
    \end{equation*}
    where we used that $K_\text{lin}$ is Lipschitz. For $I_{12}$ we start with a similar strategy,
    \begin{align*}
        &\sum_{m=1}^{i} I_{12, i} = \underbrace{\int_\Omega (K_\text{lin}(h_i) - K_\text{lin}(\Hat{h}(t_i)))  \nabla \Bar{\Theta}_i \cdot \nabla \left(\frac{\hat{\Theta}(t_i) + \hat{\Theta}(t_{i-1})}{2}\right) \di{x}}_{= I_{12, a}}
        \\ 
        &+ \underbrace{\sum_{m=1}^{i-1} \int_\Omega (K_\text{lin}(h_m)  - K_\text{lin}(\hat{h}(t_{m}))) \nabla \Bar{\Theta}_m  \cdot \nabla \left(\frac{\Hat{\Theta}(t_{m}) + \Hat{\Theta}(t_{m-1})}{2} - \frac{\Hat{\Theta}(t_{m+1}) + \Hat{\Theta}(t_{m})}{2}\right) \di{x}}_{= I_{12, b}}
        \\
        &+ \underbrace{\sum_{m=1}^{i-1} \int_\Omega \left[K_\text{lin}(h_m)  - K_\text{lin}(\hat{h}(t_{m})) - K_\text{lin}(h_{m+1}) + K_\text{lin}(\hat{h}(t_{m+1}))\right] \nabla \Bar{\Theta}_m  \cdot \nabla \left(\frac{\Hat{\Theta}(t_{m}) + \Hat{\Theta}(t_{m-1})}{2}\right) \di{x}}_{= I_{12, c}}.  
    \end{align*}
    The first term $I_{12, a}$ can be estimated with the same ideas as before, leading to
    \begin{equation*}
        |I_{12, a}| \leq \e \|\nabla \Bar{\Theta}_i\|_{L^2(\Omega)} + C (\Delta t )^2 + C \sum_{n=1}^{i-1} \Delta t \|\Bar{\Theta}_n\|_{L^2(\Omega)}^2. 
    \end{equation*}
    For $I_{12, b}$ we need \cref{eq:proof_h_estimate} and that the time derivative of $\nabla \hat{\Theta}$ exists, to get the estimate
    \begin{align*}
        |I_{12, b}| &\leq C \sum_{m=1}^{i-1} \Delta t \|\nabla \Bar{\Theta}_m\|_{L^2(\Omega)} \left[\Delta t  + \sum_{n=1}^{m-1} \Delta t \|\Bar{\Theta}_n\|_{L^2(\Omega)} \right] 
        \\
        &= C \sum_{m=1}^{i-1} \Delta t \|\nabla \Bar{\Theta}_m\|_{L^2(\Omega)} \left[\Delta t + (i-1 - m) \Delta t \|\Bar{\Theta}_m\|_{L^2(\Omega)} \right] 
        \\
        &\leq C (\Delta t)^2 + C \sum_{m=1}^{i-1} \Delta t \| \Bar{\Theta}_m\|_{H^1(\Omega)}^2.
    \end{align*}
    To bound $I_{12, c}$ we use that $K_\text{lin} \in H^1((-a^*, a^*))$ and since the exact $K$ is in $C^{1,1}$ we also have $|\partial_h K_\text{lin}(h_1) - \partial_h K_\text{lin}(h_2)| \leq C |h_1 - h_2| + \Delta h$, for a.a. $h_1, h_2 \in (-a^*, a^*)$. This yields
    \begin{align*}
        |I_{12, c}| &\leq C \sum_{m=1}^{i-1} \Delta t \int_\Omega \|\nabla \Bar{\Theta}_m\| \int_0^1 \Bigl|\partial_h K(h_m + \tau (h_{m+1} - h_{m})) 
        \\
        &\hspace{230pt} - \partial_h K(\hat{h}(t_{m}) + \tau (\hat{h}(t_{m+1}) - \hat{h}(t_{m})))\Bigr| \di{\tau} \di{x}
        \\
        &\leq C \sum_{m=1}^{i-1} \Delta t \int_\Omega \|\nabla \Bar{\Theta}_m\| \int_0^1 \left |\Delta h + h(t_m) - h_m + \tau \left[h_{m+1} - h_{m} - \hat{h}(t_{m+1}) + \hat{h}(t_{m})\right] \right | \di{\tau} \di{x}
        \\
        &\leq C \sum_{m=1}^{i-1} \Delta t \|\nabla \Bar{\Theta}_m\|_{L^2(\Omega)} \left [ \Delta h + \Delta t + \sum_{n=1}^{m-1} \Delta t \|\Bar{\Theta}_n\|_{L^2(\Omega)} \right].
    \end{align*}
    Next, we apply the Young-inequality and can proceed in the same way as for the estimate of $I_{12, a}$. Lastly, $I_{13}$ can again be estimated by summation, leading to
    \begingroup
    \allowdisplaybreaks
    \begin{align*}
        \left | \sum_{m=1}^i  I_{13, i} \right | 
        &\leq \int_\Omega \left | K_\text{lin}(h(t_i)) \nabla \left(\frac{\Hat{\Theta}(t_{i-1}) - \Hat{\Theta}(t_{i})}{2}\right) \nabla \Bar{\Theta}_i \di{x} \right | 
        \\
        &+ \sum_{m=1}^{i-1} \left | \int_\Omega (K_\text{lin}(h(t_m))  - K_\text{lin}(h(t_{m+1}))) \nabla \Bar{\Theta}_m  \cdot \nabla \left(\frac{\Hat{\Theta}(t_{m-1}) - \Hat{\Theta}(t_{m})}{2}\right) \di{x} \right | 
        \\
        &+ \sum_{m=1}^{i-1} \left | \int_\Omega K_\text{lin}(h(t_{m+1})) \nabla \Bar{\Theta}_m  \cdot \left(\frac{\Hat{\Theta}(t_{m-1}) - \Hat{\Theta}(t_{m})}{2} - \frac{\Hat{\Theta}(t_{m}) - \Hat{\Theta}(t_{m+1})}{2} \right) \di{x} \right | 
        \\
        &\leq C \Delta t \|\nabla \Bar{\Theta}_i\|_{L^2(\Omega)} + \sum_{m=1}^{i-1} C \|\nabla \Bar{\Theta}_m\|_{L^2(\Omega)} (\Delta t )^2
        \\
        &\leq C (\Delta t)^2 + \e \|\nabla \Bar{\Theta}_i\|_{L^2(\Omega)}^2 + \sum_{m=1}^{i-1} \Delta t \|\nabla \Bar{\Theta}_m\|_{L^2(\Omega)}^2.
    \end{align*}
    \endgroup
    In the first inequality above the Lipschitz continuity of $K_\text{lin}$, the Hölder inequality, and the time-regularity of $\nabla \hat{\Theta}$ were utilized. In the second inequality, we used Young again.

    We can now combine all of the above estimates. To this end, we also have to sum $I_1, \dots I_{10}$ over $1,\dots,i$ to combine them with the gradient estimates, to be precise we would first sum and then carry out the above computations inside the sum. This leads to the estimate
    \begin{align}\label{eq:proof_first_estimate}
            \sum_{m=1}^i &\Delta t C(a^*) \|d_t \Bar{\Theta}_m\|^2_{L^2(\Omega)} + K(a^*) \|\nabla \Bar{\Theta}_i\|_{L^2(\Omega)}^2 
            \leq C (\Delta t)^2 +  
            C \e \|\nabla \Bar{\Theta}_i\|_{L^2(\Omega)}^2
            + \e \sum_{m=1}^i \Delta t C \|d_t \Bar{\Theta}_m\|^2_{L^2(\Omega)} \notag
            \\
            &+ C \sum_{m=1}^{i-1} \Delta t \left[ \|\Bar{\Theta}_m\|^2_{H^1(\Omega)} + \|\Bar{\Theta}_m\|^2_{L^2(\Omega)} + \|\nabla \Bar{\Theta}_m\|^2_{L^2(\Omega)} \right]
            + C \sum_{m=1}^{i} \Delta t \sum_{n=1}^{m-1} \Delta t \|\Bar{\Theta}_n\|^2_{L^2(\Omega)} +  C (\Delta h)^2. 
    \end{align}
    We are missing the $L^2$-norm of $\Bar{\Theta}_i$ on the left for the final estimate. Therefore, we use the same computation as \cite[Eq. (3.19)]{Ewing1978} to obtain
    \begin{equation}\label{eq:proof_add_theta}
        \begin{split}
            K(a^*)\left(\|\bar{\Theta}_i\|^2_{L^2(\Omega)} + \|\bar{\Theta}_{i-1}\|^2_{L^2(\Omega)}\right) 
            &\leq 2 K(a^*) \Delta t \int_\Omega d_t \Bar{\Theta}_i \Bar{\Theta}_{i-1} \di{x} + K(a^*) (\Delta t)^2 \|d_t \bar{\Theta}_i\|^2_{L^2(\Omega)}
            \\
            &\leq \frac{C(a^*)\Delta t}{4} \|d_t \bar{\Theta}_i\|^2_{L^2(\Omega)} + C \Delta t \|\bar{\Theta}_{i-1}\|^2_{L^2(\Omega)}.
        \end{split}
    \end{equation}
    For the last inequality, we used the assumption $\Delta t \leq \nicefrac{C(a^*)}{8K(a^*)}$. Next, we sum \cref{eq:proof_add_theta} from $1,\dots,i$ and add the estimate to \cref{eq:proof_first_estimate}. Then the sum over $\| d_t \bar{\Theta}\|$ which originates when summing  (\ref{eq:proof_add_theta}) can be collected into the already present term in (\ref{eq:proof_first_estimate}). 
    Additionally, the $\e$ on the right side originates from the Young inequality and can be chosen such, that all terms connected to $\e$ can be incorporated into the left-hand side. All other sums on the right can be further bounded by the $H^1$-norm, which yields
    \begin{align*}
        \sum_{m=1}^i \Delta t C \|d_t \Bar{\Theta}_m\|^2_{L^2(\Omega)} + C \|\Bar{\Theta}_i\|_{H^1(\Omega)}^2 
        \leq C (\Delta t)^2 + C (\Delta h)^2 + C \sum_{m=1}^{i-1} \Delta t \|\Bar{\Theta}_m\|^2_{H^1(\Omega)}.
    \end{align*}
    Finally, applying the discrete Grönwall lemma gives the desired result. As mentioned at the start the estimates for $\Bar{\vartheta}$ follow analogously. Plugging the estimate for $\Bar{\Theta}$ back into \cref{eq:proof_h_estimate}, we find the estimate for $\Bar{h}$. 
\end{proof}
\begin{remark}
    The convergence of the time-stepping scheme requires that the precomputing also gets refined with smaller time steps. This condition appears because the linear interpolation has a piecewise constant derivative. This is not necessary for a smoother interpolation, for example, quadratic splines.
\end{remark}
\subsection{Further error sources} In addition to the discretization in time, we also have to discretize the function space.
To this end, we consider a finite element (FE) approach. 
Define the triangulation $\mathcal{T}_{H_M,\Omega}$ of $\Omega$ and $\mathcal{T}_{H_m,Y_0}$ of $Y_0$, where $H_M$ and $H_m$ denote the largest edge length of the simplices in $\mathcal{T}_{H_M,\Omega}$ and $\mathcal{T}_{H_m,Y_0}$ respectively.
Denote with $\mathbb{P}_1$ the space of linear polynomials. On the macro scale, we define the finite element space $\mathrm{X}_{H_M,\Omega} \subset H^1(\Omega)$ by
\begin{equation*}
    \mathrm{X}_{H_M,\Omega} = \left\{
        \Phi \in C(\overline{\Omega}) : \Phi_{|T} \in \mathbb{P}_1(T) 
        \text{ for all } T \in \mathcal{T}_{H_M,\Omega}
    \right\}.
\end{equation*}
Similarly, on the microscale the space $\mathrm{X}_{H_m, Y_0} \subset H^1(Y_0)$ is given by
\begin{equation*}
    \mathrm{X}_{H_m, Y_0} = \left\{
        \phi \in C(\overline{Y_0}) : \phi_{|T} \in \mathbb{P}_1(T) 
        \text{ for all } T \in \mathcal{T}_{H_m,Y_0}
    \right\}.
\end{equation*}
Let $N_M$ denote the degrees of freedom (DOFs) on the macroscale and $N_m$ DOFs on the microscale. Define the basis functions $\{\Phi_n\}_{n=1}^{N_M}$ such that $\text{span}(\Phi_n) = \mathrm{X}_{H_M,\Omega}$ and equally the micro basis functions $\{\phi_k\}_{k=1}^{N_m}$. The time  discretized finite element approximation of the solutions $(\Theta, \vartheta, h)$ are then given by ($H=(H_M,H_m)$)
\begin{align*}
    \Theta_{i, H} = \sum_{n=1}^{N_M} \alpha_{i, n} \Phi_n(x),
    \quad 
    \vartheta_{i, H} = \sum_{n=1}^{N_M} \sum_{k=1}^{N_m}  \beta_{i, n, k} \Phi_n(x) \phi_k(y) 
    \quad \text{ and } \quad 
    h_{i, H} = \sum_{n=1}^{N_M} h_{i, n} \Phi_n(x),
\end{align*}
with the additional condition $\Theta_{i, H} = \vartheta_{i, H \, |\Gamma_0}$ for all $i=1,\dots,N$. For the coefficients it holds, $\alpha_{i, n}, \beta_{i, n, k}, h_{i, n} \in \R$.

The above FE approach introduces another error in the numerical approximation. The error estimate from \cref{lemma:time_stepping} can also be extended to include the projection error in regards to the FE space. This was carried out for a comparable problem in \cite[Section 3]{Ewing1978} and could be transferred to our setup under regularity assumptions of the solutions. Here, we do not carry out the complete proof, instead, we demonstrate the error in the following section through numerical simulations. 

Note that the above FE approach is a natural candidate for discretizing the problem, but can become quite expensive due to the high dimensionality of $\vartheta$, especially in three dimensions. We still use the above method since we only want to demonstrate the previous analytical results on a two-dimensional setup. Parallelization is used to speed up the calculations; see \cref{sec:simulations} for details. However, in other applications, this procedure may not be feasible. In \cref{subsec:lit_review} we gave an overview of possible extensions and different approaches. Note, that the presented precomputing method and theory could also be combined with the alternatives to reduce the computational effort.

Next to the discretization of $(\Theta, \vartheta)$ we also need to solve the cell problems \cref{eq:system:cell_problem} with a FE approach. To this end introduce the mesh size $H_\text{cell}$ of the mesh for $Y\setminus Y(h)$. Let $K_H$ denote the effective conductivity computed with the FE solution. Under the assumption that we use piecewise linear FEs for solving the cell problems and the cell solutions are sufficiently smooth, we can obtain the error estimate
\begin{equation*}
    |K(h) - K_H(h)| \leq C H_\text{cell},
\end{equation*}
for a $C < \infty$ independent of $H_\text{cell}$.
\section{Simulations}\label{sec:simulations}
In this last section, we want to investigate and back up the previous analysis results with numerical simulations. To this end, we apply the time-discretized, finite element approach with precomputation as explained in \cref{sec:numerical_analysis}. Before we study the convergence behavior, we introduce the general problem setup and explain some implementation details.
We use the FEM library FEniCS \cite{Fenics} and implement a custom parallelization for the micro problems. Our code is freely available online \cite{Code}. 

As mentioned in the previous section, we want to apply an approach where the time step for $\Theta$ and $\vartheta$ is carried out simultaneously. Therefore, at time step $i$, we have to solve the linear system
\begin{equation*}
    \begin{pmatrix}
        A & C_1 & C_2 &\dots & C_{N_M} 
        \\
        D_1 & B_1 & 0 & \dots & 0 
        \\
        D_2 & 0 & \ddots & & \vdots
        \\
        \vdots & \vdots & & \ddots & 0
        \\
        D_{N_M} & 0 & \dots & 0 & B_{N_M} 
    \end{pmatrix}
    \begin{pmatrix}
        \alpha_i
        \\
        \beta_{i, 1} 
        \\
        \vdots
        \\
        \vdots
        \\
        \beta_{i, N_M}
    \end{pmatrix}
    = 
    \begin{pmatrix}
        F_i
        \\
        f_{i, 1} 
        \\
        \vdots
        \\
        \vdots
        \\
        f_{i, N_M}
    \end{pmatrix}
    .
\end{equation*}
Here, we do not want to go too deep into the implementation details of each matrix block, for this, we point to the available source code. Instead, we only explain the role of each part:
\begin{itemize}
    \item The matrix $A \in \R^{N_M \times N_M}$ encodes the macroscopic bilinear form.
    \item The matrices $B_k \in \R^{N_m \times N_m}$ represent the microscopic bilinear form, for $k=1,\dots N_M$.
    \item $C_k \in \R^{N_M \times N_m}$ describes the heat flow across the micro boundary $\Gamma_0$, and can be created by computing $\kappa(h_i) \nabla \vartheta_i \cdot n$ over $\Gamma_0$.
    \item $D_k \in \R^{N_m \times N_m}$ enforces the coupling that $\Theta = \vartheta$ on $\Gamma_0$. Therefore this matrix only contains 0 and 1 on the corresponding entries. In particular, the row structure of all $D_k$ is the same, only one column contains ones and for each $k$ this column is shifted to a different position. In the corresponding rows of matrix $B_k$ only a $-1$ is present and on the right-hand side a 0.
    \item The right-hand side $F_i \in \R^{N_M}$ and $ f_{i, k} \in \R^{N_m}$ include the heat sources and information from the previous time step.
\end{itemize}
The role of $C_k$ and $D_k$ is the coupling between macro and micro scale. Note that there is no direct coupling between the different cells and the block containing the $B_k$ matrices is diagonal. Additionally, all matrices are sparse.
The matrices $A, B_k, C_k$ need to be recomputed in each time step, since they depend on the current height $h_i$. This is easily parallelizable once $h_i$ is known.

Before we start presenting the simulation results, we collect here all information regarding the data and geometry. For the macro domain we use $\Omega = [0, 1]^2$, the time interval is set to $S=[0, 10]$, and the microdomain $Y_0$ is defined as a disk with radius $r = 0.25$. Note, that for a disk the Hanzawa transformation only is a scaling in the direction of the radius. The conductivity values are set to $\kappa = \mathcal{K} = 0.1$ and the initial temperature to $\Theta_0 = \vartheta_0 = \Theta^{ref} = 0$. We only consider a macroscopic heat source, meaning $f = 0$, given by
\begin{equation*}
    F(x, t) = 0.75 \max \Bigl\{0,  \min \bigl[1, g(t, x)\bigr]\Bigr\} \max\Bigl\{0, \min \bigl[1, 5 - t\bigr]\Bigr\}
\end{equation*}
with
\begin{equation*}
    g(t, x) = 2 - 10 \left\|x - \begin{pmatrix}
         0.2 + \nicefrac{0.6 t}{5}   \\
         0.7 
    \end{pmatrix} \right\|_\infty.     
\end{equation*}
$F$ describes a smoothed step function that moves to the right and becomes 0 for $t \geq 5$.
Additionally, we introduce a speed variable $v=0.1$ to better control the change of the microdomain, e.g., $\partial_t h = v(\Theta - \Theta^{ref})$. This does not influence the previous analysis. We remark, that we do not have a concrete application in mind for this setup. Instead, the values are chosen to observe interesting aspects of the model.

If not mentioned otherwise, we use $\Delta t = 0.1$, $H_M = 0.05$, $H_m = 0.06$, $H_\text{cell} = 5\mathrm{e}{-4}$ and $\Delta h = \nicefrac{1}{320}$ in the following simulations. 

Since our microdomains are symmetric around the center of the domain, the effective conductivity $K(h)$ reduces to a scalar value. The dependence of $K$ of different radii $r$, is shown in \cref{fig:K_dependence_of_h}. Note, that we characterize the change in radius by the height function $h$, a negative $h$ equals shrinkage and a positive $h$ growth of the disk radius. 
\begin{figure}[ht]
    \centering
    \begin{subfigure}[b]{0.5\textwidth}
         \centering
         \includegraphics[width=\textwidth]{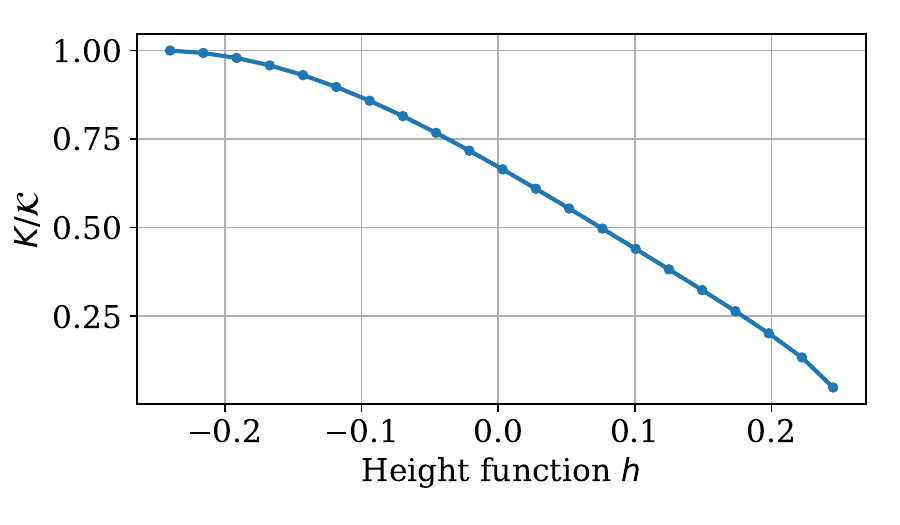}
     \end{subfigure}
     \hfill
     \begin{subfigure}[b]{0.24\textwidth}
         \centering
         \includegraphics[width=\textwidth]{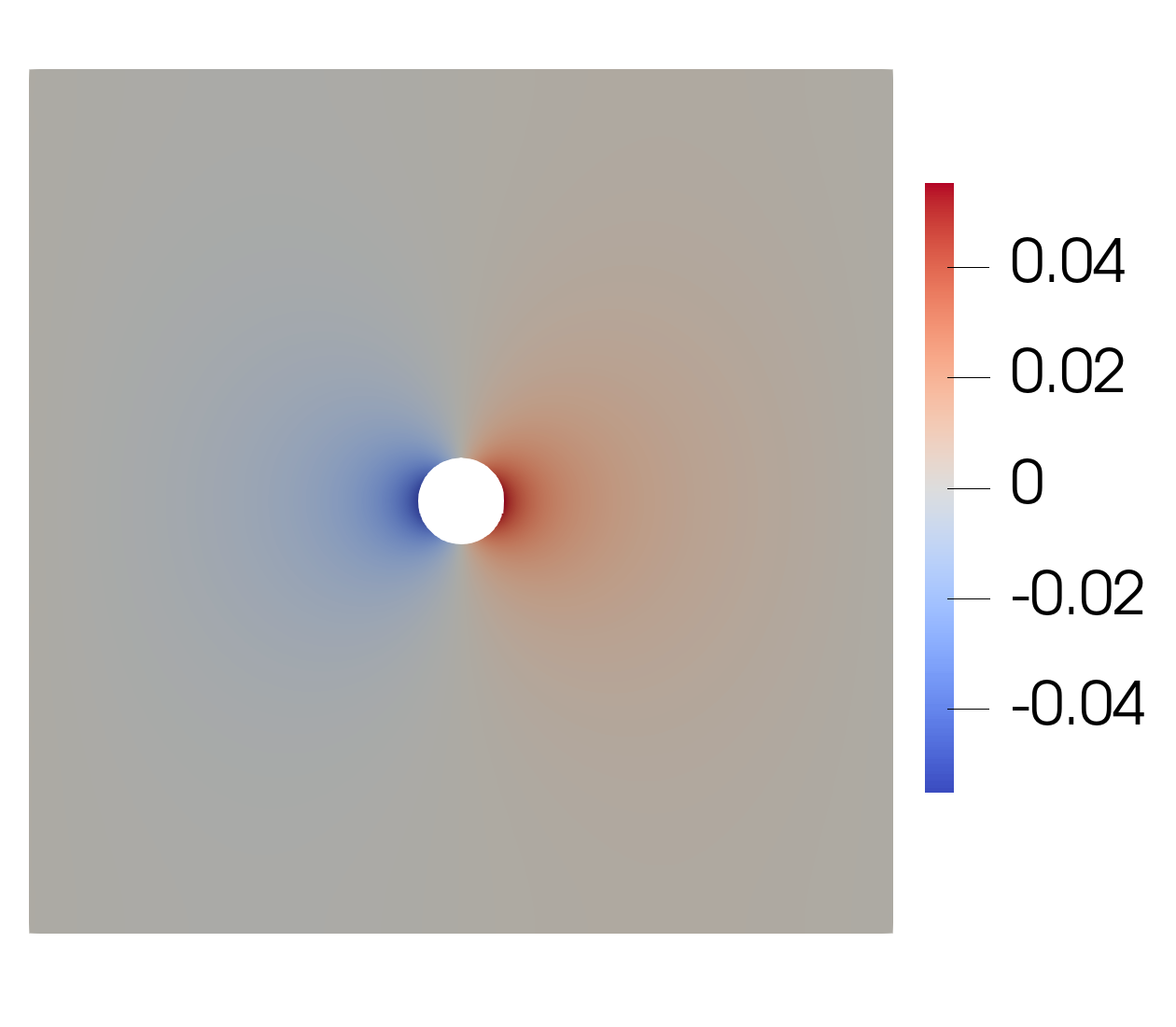}
         \vspace{0.1cm}
         \caption{$h=-0.2$}
     \end{subfigure}
     \hfill
     \begin{subfigure}[b]{0.24\textwidth}
         \centering
         \includegraphics[width=\textwidth]{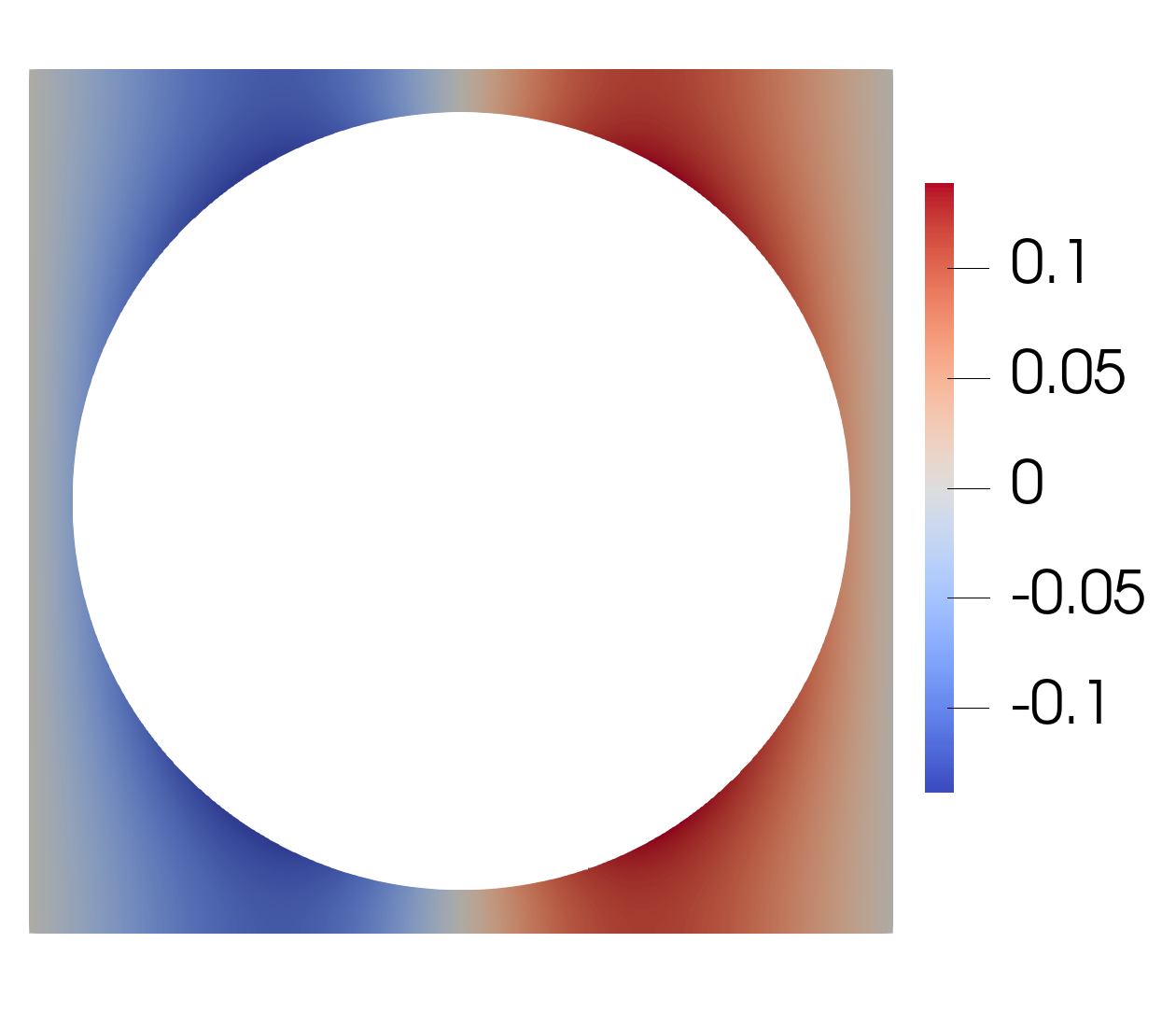}
         \vspace{0.1cm}
         \caption{$h=0.2$}
     \end{subfigure}
    \caption{The effective conductivity $K$ for changes in height $h$. At the right two cell solutions for different $h$.}
    \label{fig:K_dependence_of_h}
\end{figure}
\subsection{Precomputing strategy}
We start with the numerical investigation of the precomputing method. Both a piecewise linear interpolation as well as a quadratic spline interpolation will be studied. For the precomputation, we solve $(\Delta h)^{-1} = 320$ cell problems for different height values $h$ that are taken from a uniform grid in $[-0.245, 0.245]$. Outside of this interval, we extend the interpolation by a constant value. Since no analytic solution is at hand, we compute a reference solution that uses all cell problems and the quadratic spline interpolation. The error is then computed between the reference solution and simulation results that utilize fewer cell problems for the interpolation.

The results are presented in \cref{fig:precomputing}. For the linear interpolation, we can observe the expected convergence order of $\mathcal{O}\left((\Delta h)^2\right)$. In the case of quadratic interpolation, we note for a large $\Delta h$, meaning the use of fewer precomputed values, we detect that the error decreases with the expected order of $\mathcal{O}\left((\Delta h)^3\right)$. However, as soon as finer interpolation is used, the error stops decreasing and reaches a plateau. This is where the discretization error of the cell problems comes into play, affecting the accuracy of the interpolation. This effect is further described for different $H_\text{cell}$ in \cref{fig:precomputing_c}. There we focus only on the function $\Theta$ to simplify the visualization. It can be seen that for larger $H_\text{cell}$, meaning a coarser cell mesh, the plateau occurs earlier. Therefore, to obtain the expected convergence in the precomputing scheme, an accurate solution of the cell problems is mandatory.
\begin{figure}[ht]
    \centering
    \begin{subfigure}[b]{0.32\textwidth}
         \centering
         \includegraphics[width=\textwidth]{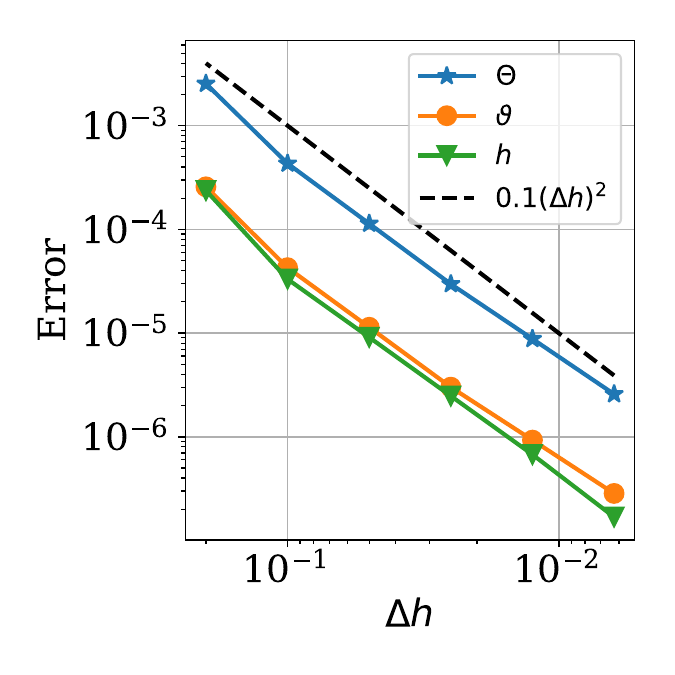}
         \caption{Linear interpolation}
     \end{subfigure}
     \hfill
     \begin{subfigure}[b]{0.32\textwidth}
         \centering
         \includegraphics[width=\textwidth]{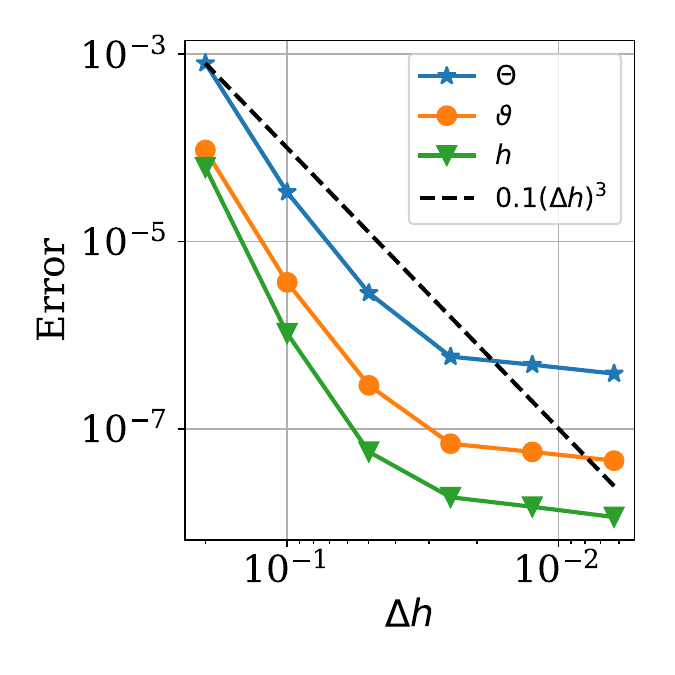}
         \caption{Quadratic interpolation}
     \end{subfigure}
     \hfill
     \begin{subfigure}[b]{0.32\textwidth}
         \centering
         \includegraphics[width=\textwidth]{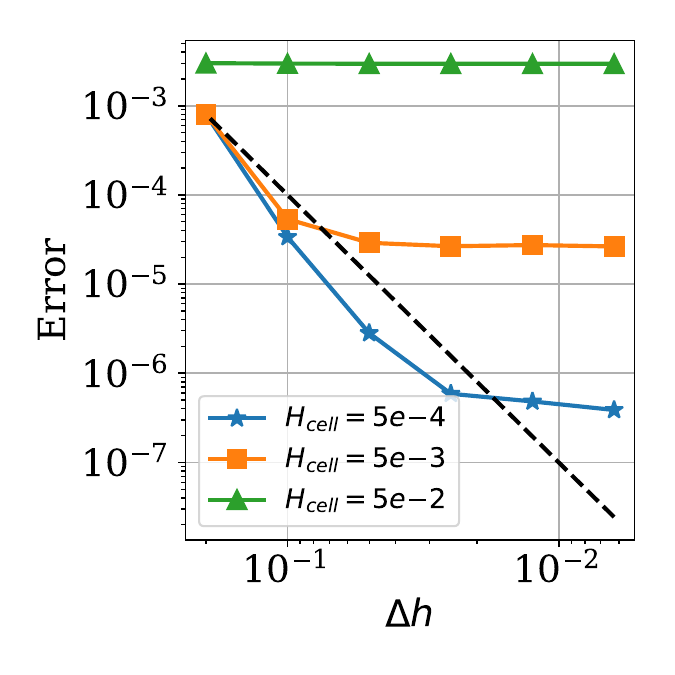}
         \caption{Dependence on $H_\text{cell}$}
         \label{fig:precomputing_c}
     \end{subfigure}
    \caption{Convergence curves for the precomputing strategy. (a) shows the error for linear interpolation of $K$, (b) the error for quadratic spline interpolation of $K$. For $\Theta$ the plotted error is in the $L^2(S;H^1(\Omega))$ norm, for $\vartheta$ the $L^2(S\times \Omega; H^1(Y_0))$ norm is used and the error for $h$ is computed in the $L^2(S\times \Omega)$ norm. In (c) the influence of discretization $H_\text{cell}$ of the cell solutions on the accuracy of the quadratic interpolation is demonstrated. Here, only the error for $\Theta$ is shown.}
    \label{fig:precomputing}
\end{figure}
Another aspect we want to highlight is the interpolation property when the disk gets close to the outer boundary of the unit cell, meaning $h \to 0.25$ or equally $r \to 0.5$. Corresponding to the previous theory, for a radius close to $0.5$ we would obtain a small $a^*$ since the disk is not allowed to grow much further. Therefore, by \cref{remark:Lipschitz_constant_of_derivative} the slope of $K(h)$ may become large. This is also observable in \cref{fig:K_dependence_of_h}. To better demonstrate the effect on the interpolation, we construct linear and quadratic interpolations for different $\Delta h$ and compare them on the interval $I_h = [-0.25, h_\text{max}]$. Again as a reference, we use the finest interpolation with $\Delta h = \nicefrac{1}{320}$ and compare by computing the $L^\infty$-difference on a uniform grid in $I_h$ with 1000 points. This comparison is visualized in \cref{fig:interpolation_at_boundary}. We notice, that when $h_\text{max}$ gets closer $0.25$ the convergence order gets reduced. To circumvent this effect one could employ a nonuniform sampling of the height values used for the interpolation, e.g., solve more cell problems when $h$ is close to $0.25$. Note, that in the quadratic spline interpolation, we again observe the influence of $H_\text{cell}$ once $\Delta h$ is small. From a theoretical point of view, this effect could also appear when $h \to 0$, but there we observe that the slope of $K(h)$ is rather flat. 

Finally, we want to briefly mention the computation time. On our setup, an AMD EPYC 7282 16-core CPU, it takes about 10 minutes to solve a cell problem for $H_\text{cell}=5\mathrm{e}{-4}$, including the mesh generation. Since the precomputation is easily parallelizable, it took about 5 hours to compute the full dataset of 320 cell problems, as we were able to spread the work over several CPUs. Then, solving the entire problem in parallel on 8 cores, using the discretization mentioned above, takes about 13 seconds. If we had not used the precomputation scheme, the cell problems would need to be solved at each macro DOF and possibly at each time step, this would result in computation times of several days or weeks. This is also the reason why we do not use such a solution as a reference for the convergence study. 
\begin{figure}[ht]
    \centering
    \begin{subfigure}[b]{0.48\textwidth}
         \centering
         \includegraphics[width=\textwidth]{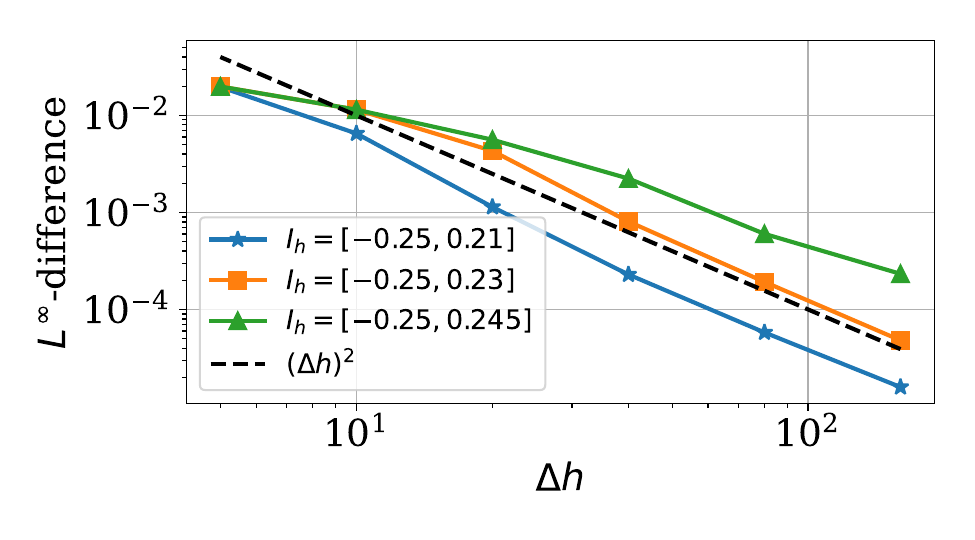}
         \caption{Linear interpolation}
     \end{subfigure}
     \hfill
     \begin{subfigure}[b]{0.48\textwidth}
         \centering
         \includegraphics[width=\textwidth]{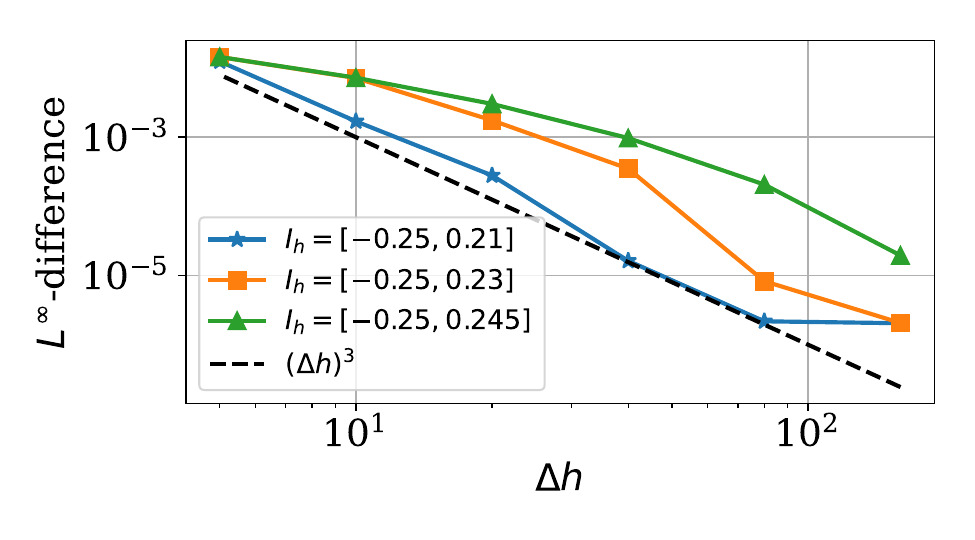}
         \caption{Quadratic interpolation}
     \end{subfigure}
    \caption{Convergence order of the interpolation of $K$ depending on the size of the height interval $I_h$. }
    \label{fig:interpolation_at_boundary}
\end{figure}
\subsection{Time and space discretization error}
In this second simulation section, we briefly present the convergence behavior with respect to the spatial discretization of the macro and micro scales and the time discretization. We study the influence of each discretization independently and fix the other parameters in each simulation. As before, no analytical solutions are available and a comparison is made with a finer numerical solution. For the fine reference solution, we use $H_M=0.01$, $H_m=0.006$, and $\Delta t = 0.003$ respectively for each convergence study and apply quadratic spline interpolation for the precomputing.

The results are presented in \cref{fig:error_convergence}. The order of error for the time-stepping method is consistent with the theoretical results of \cref{lemma:time_stepping}. Given the FE spaces of piecewise linear polynomials, both on the macro and micro scale, we also obtain the expected convergence behavior in regard to the space discretization.
\begin{figure}[ht]
    \centering
    \includegraphics[width=\textwidth]{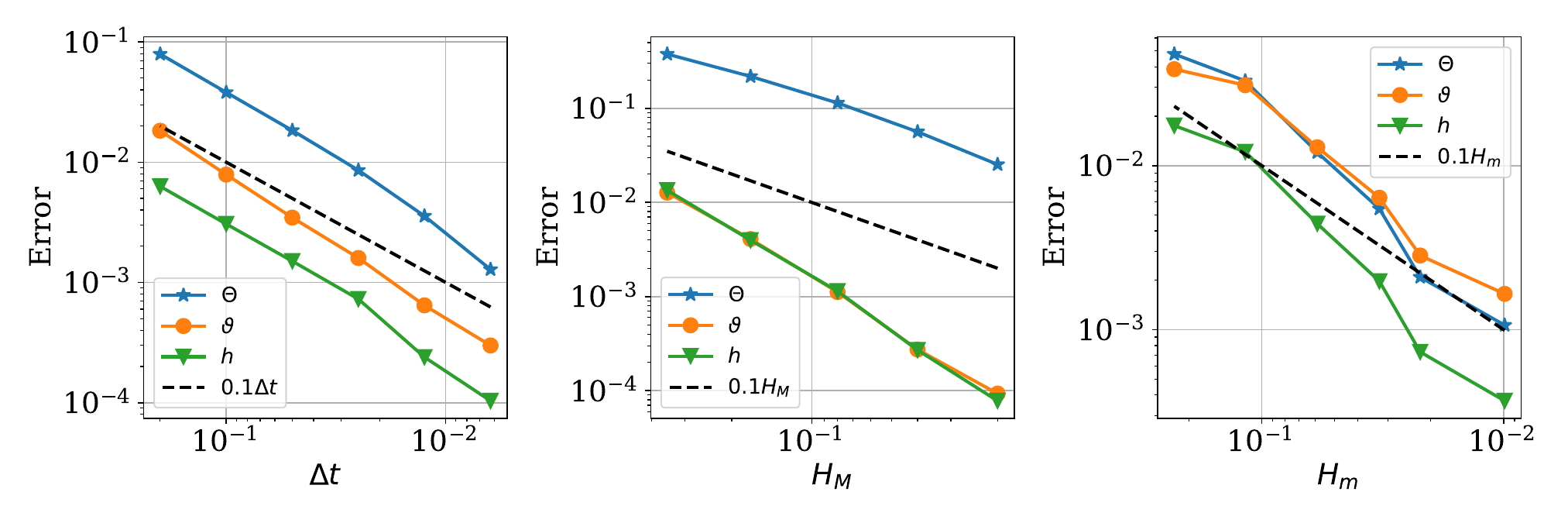}
    \caption{Convergence study concerning the time and space discretization. At the left: the behavior for different time steps $\Delta t$, in the middle: for different mesh sizes $H_M$ on the macro domain and on the right: different resolutions on the microdomain. For $\Theta$ the plotted error is the $L^2(S;H^1(\Omega))$ error, for $\vartheta$ the $L^2(S\times \Omega; H^1(Y_0))$ norm is used and the error for $h$ is computed in the $L^2(S\times \Omega)$ norm.}
    \label{fig:error_convergence}
\end{figure}
\section{Concluding remarks}\label{section:conclusion}
We have suggested, analyzed, and implemented a precomputing strategy for non-linearly coupled two-scale systems with evolving microstructures.
Utilizing the Hanzawa transformation, the existence of local-in-time, weak solutions was established.
Under more restrictive conditions on the data, it is expected that the solutions are unique and strong (see also \cite{Chen92,GaerttnerKnabnerRay23}).
We showed that the additional interpolation error introduced via the precomputing can be controlled by the step size and the interpolation order given that the underlying relationship in the material parameters is smooth.
We also introduced a semi-implicit time-stepping method to resolve the nonlinearities inherent to the moving microsctructures and showed linear convergence.
We note that it is possible to improve convergence by using, e.g., a Crank-Nicolson-Galerkin
approximation (\cite{Ewing1978}).
The simulation results presented in \cref{sec:simulations} confirm the analytical convergence.
The implementation was done in FEniCS and our code is freely available online \cite{Code}.

While the model in question in this paper is concerned with phase transitions on the microscale, we want to point out that the same kind of model with very similar analytical and numerical challenges appears in many different applications: precipitation in porous media, swelling of biological tissues, and degradation of concrete to name just a few.
The main advantage of our approach is the fact that a significant part of the computational effort is pushed into an offline phase in a way that is additionally \textit{perfectly parallelizable}.
There are of course potential drawbacks as well: When the underlying functions are discontinuous or extremely volatile, the interpolation error might be difficult to control.
Moreover, in our case we only have a one-dimensional bounded parameter space $h\in[-a^*,a^*]$, which is relatively easy to exhaust.
In cases where shape changes are possible (e.g.,~\cite{GaerttnerKnabnerRay23,LichtiLeichner2022}) the parameter space becomes much larger.

\section*{Acknowledgements}
The research activity of ME is funded by the European Union’s Horizon 2022 research and innovation program under the Marie Skłodowska-Curie fellowship project {\em{MATT}} (funding number 101061956, \url{https://doi.org/10.3030/101061956}).

TF acknowledges funding by the Deutsche Forschungsgemeinschaft (DFG, German Research Foundation) -- project nr.  
281474342/GRK2224/2 and by project nr. 439916647.

\bibliographystyle{abbrv}
\bibliography{references}

\end{document}